\documentclass[a4paper, 11pt, parskip=half-]{scrartcl}
\pdfoutput=1

\usepackage[thmmarks]{ntheorem}
\usepackage[utf8]{inputenc}
\usepackage{amsmath, amsfonts, amssymb}
\usepackage[scr=boondoxo]{mathalfa}
\usepackage[english]{babel}
\usepackage{csquotes}
\usepackage[backend=biber,style=alphabetic,uniquename=true,uniquelist=true]{biblatex}
\usepackage{graphicx}
\usepackage[colorlinks, linkcolor={cyan}, citecolor={cyan}, breaklinks=true]{hyperref}
\usepackage{tikz}
\usepackage{tikz-cd}
\usepackage{svg}
\usepackage[export]{adjustbox}
\usepackage[sc]{mathpazo}
\usepackage{microtype}
\usepackage{enumitem}

\numberwithin{equation}{section}
\theoremstyle{plain}
\theorembodyfont{\itshape}
\theoremheaderfont{\normalfont\sffamily\bfseries}
\theorembodyfont{\normalfont}
\theorempreskip{0.5\baselineskip}
\newtheorem*{itheorem*}{Theorem}
\theoremseparator{.}

\newcounter{thm}
\numberwithin{thm}{subsection}

\newtheorem{theorem}[thm]{Theorem}
\newtheorem{corollary}[thm]{Corollary}
\newtheorem{lemma}[thm]{Lemma}
\newtheorem{proposition}[thm]{Proposition}
\newtheorem{definition}[thm]{Definition}
\newtheorem{example}[thm]{Example}
\newtheorem{remark}[thm]{Remark}

\theoremstyle{nonumberplain}
\theoremheaderfont{\itshape}
\theoremseparator{.\,—}
\theoremsymbol{\ensuremath{\square}}
\newtheorem{proof}{Proof}

\title{Three-Dimensional Spin TFTs\\ from Gauging Line Defects}
\author{Jannik Gröne \qquad Ingo Runkel\\[0.5cm]
	\normalsize{\texttt{\href{mailto:jannik.groene@uni-hamburg.de}{jannik.groene@uni-hamburg.de}}} \\
	\normalsize{\texttt{\href{mailto:ingo.runkel@uni-hamburg.de}{ingo.runkel@uni-hamburg.de}}}\\[0.5cm]
	\normalsize\slshape Fachbereich Mathematik, Universität Hamburg,\\
	\normalsize\slshape Bundesstraße 55, 20146 Hamburg, Germany}
\date{}

\DeclareFieldFormat{doilink}{\iffieldundef{doi}{#1}{\href{http://dx.doi.org/\thefield{doi}}{#1}}}
\DeclareBibliographyDriver{article}{%
  \usebibmacro{bibindex}%
  \usebibmacro{begentry}%
  \usebibmacro{author/translator+others}%
  \setunit{\labelnamepunct}\newblock
  \usebibmacro{title}%
  \newunit
  \printlist{language}%
  \newunit\newblock
  \usebibmacro{byauthor}%
  \newunit\newblock
  \usebibmacro{bytranslator+others}%
  \newunit\newblock
  \printfield{version}%
  \newunit\newblock
  \usebibmacro{in:}%
  \printtext[doilink]{%
  \usebibmacro{journal+issuetitle}%
  \newunit
  \usebibmacro{byeditor+others}%
  \newunit
  \usebibmacro{note+pages}%
  }%
  \newunit\newblock
  \iftoggle{bbx:isbn}
    {\printfield{issn}}
    {}%
  \newunit\newblock
  \usebibmacro{doi+eprint+url}%
  \newunit\newblock
  \usebibmacro{addendum+pubstate}%
  \setunit{\bibpagerefpunct}\newblock
  \usebibmacro{pageref}%
  \usebibmacro{finentry}}

\DeclareMathOperator{\id}{\mathrm{id}}
\DeclareMathOperator{\im}{\mathrm{im}}
\DeclareMathOperator{\Obj}{\mathrm{Obj}}

\newcommand\Qb{\mathbb{Q}}
\newcommand\Rb{\mathbb{R}}
\newcommand\Cb{\mathbb{C}}
\newcommand\Zb{\mathbb{Z}}
\newcommand\Hb{\mathbb{H}}
\newcommand\Cc{\mathcal{C}}
\newcommand\Fc{\mathcal{F}}
\newcommand\Rc{\mathcal{R}}
\newcommand\Dc{\mathcal{D}}
\newcommand\Zor{\mathcal{Z}^\textsf{or}}
\newcommand\Zspin{\mathcal{Z}^\textsf{sp}}
\newcommand\Vect{\textsf{Vect}}
\newcommand\sVect{\textsf{sVect}}
\newcommand\Bord{\textsf{Bord}}
\newcommand\ALineCat{\mathcal{L}_A^\textsf{sp}}
\newcommand\TargetCat{\mathcal{S}}
\newcommand\Bordsp{\textsf{Bord}^\textsf{sp}_3}
\newcommand\hBordsp{\widehat{\textsf{Bord}}^\textsf{sp}_3}

\begin{document}
\maketitle
\begin{abstract}
From the input of an oriented three-dimensional TFT with framed line defects and a commutative $\Delta$-separable Frobenius algebra $A$ in the ribbon category of these line defects, we construct a three-dimensional spin TFT. The framed line defects of the spin TFT are labelled by certain equivariant modules over $A$, and the spin structure may or may not extend to a given line defect. Physically the spin TFT can be interpreted as the result of gauging a one-form symmetry in the original oriented TFT. This spin TFT extends earlier constructions in \cite{BM96,Bla03}, and it reproduces the classification of abelian spin Chern-Simons theories in \cite{BM05}.
\end{abstract}

\newpage

{\hypersetup{linkcolor={black}}
\tableofcontents
}

\section{Introduction}
In the 1990s it was noted that constructions of Witten-Reshetikhin-Turaev invariants can split into contributions coming from different choices of spin structure. A corresponding spin TFT using skein modules and the universal construction was given in~\cite{BM96}. The refinement of the state spaces and invariants was also described via a Reshetikhin-Turaev surgery construction for $\mathfrak{sl}_2$ in~\cite{Bel98}, together with a refinement of Turaev-Viro TFTs. In~\cite{Bla03} the refinements of the Reshetikhin-Turaev TFTs were extended to all those with so-called spin modular categories as input, though no full TFT construction was given.

    Another way to construct new TFTs from known ones is via the gauging of (higher) symmetries. In this paper we will use the gauging line defects in three-dimensional oriented TFTs to produce a new TFT sensitive to a choice of spin structure on a given bordism, yielding all examples above as special cases. A construction conceptually similar to ours has been described in~\cite{GK16,HLS19}. The procedure is also known as fermionic anyon condensation in condensed matter physics. A description of the algebraic data for the condensation which is similar to ours can be found in~\cite{WW17} and a construction for state-sum like models in terms of tube categories was given in~\cite{ALW19}.

We will implement the gauging procedure using an internal state sum construction (or orbifold construction), restricted to one-dimensional strata. For two-dimensional TFTs with $r$-spin structure this has already been implemented in~\cite{NR20}. For oriented TFTs in any dimension, the internal state sum construction was developed in \cite{CRS19} and specialised to three-dimensional Reshetikhin-Turaev TFTs in \cite{CMRSS21}. Our construction generalises the commutative algebra example given there to the spin case, and it would be interesting to also generalise the full construction.

Our main result (in a slightly simplified version) is
\begin{itheorem*}{\textsf{\textbf{\!\ref{thm:zspin_is_a_TFT}.\,}}}
    Let $\TargetCat$ be an idempotent-complete symmetric monoidal category (such as $\Vect$ or $\sVect$) and let $\Cc$ be a ribbon category with $A\in\Cc$ a commutative $\Delta$-separable Frobenius algebra. Then from a symmetric monoidal functor (compatible with the skein relations for $\Cc$)
    \[
        \Zor: \Bord_3^\textsf{or}(\Cc)\to\TargetCat
    \]
    one can construct a symmetric monoidal functor
    \[
        \Zspin: \Bordsp(\ALineCat)\to\TargetCat\,.
    \]
\end{itheorem*}

In the above theorem, $\Bord_3^\textsf{or}(\Cc)$ stands for the bordism category with three-dimensional oriented bordisms and embedded $\Cc$-coloured ribbon graphs, while in $\Bordsp(\ALineCat)$ the bordisms (and surfaces) are equipped with a spin structure which may or may not be singular at the $\ALineCat$-coloured ribbons. The ribbon category $\ALineCat$ will be briefly described below. We note that $\Cc$ need not be modular: in the examples considered in Section~\ref{sec:applications}, $\Cc$ has symmetric centre $\Vect$ or $\sVect$, and $\TargetCat$ agrees with the symmetric centre. The twist $\theta_A$ of the algebra $A$ necessarily squares to $\id_A$ (see Section~\ref{sec:comm-FA}). If we already have $\theta_A = \id_A$, or equivalently if $A$ is in addition symmetric, the resulting TFT $\Zspin$ does not distinguish spin structures.

The spin TFT is constructed by inserting 1-skeleta coloured by $A$ encoding the spin structure and evaluating under the oriented TFT. The conditions on $A$ (commutative, $\Delta$-separable and Frobenius) guarantee invariance under the choice of skeleton. The prototypical (and in some sense generic, see Theorem~\ref{thm:simple_commutative_frobenius_algebras}) algebra is $A = 1\oplus f$, where $1$ is the monoidal unit and $f$ is a fermionic object in $\Cc$, i.e.\ one with ribbon twist $\theta_f = -1$ and $f\otimes f=1$. Commutativity of $A$ furthermore requires $f$ to have trivial self-braiding, and these conditions together imply that the quantum dimension of $f$ is $\dim f = -1$. If $\theta_f=-1$ but $\dim f = 1$ (as happens in unitary theories), the self-braiding is non-trivial. In this case we first extend the theory by taking the product with $\sVect$, so that the parity-shifted copy of $f$ has the required properties (see Section~\ref{sec:spin-modular-cat}).

The construction changes the category of line defects we allow to an appropriate category of modules over $A$. In the case of an oriented gauging of line defects, the resulting category is that of local modules, see e.g.~\cite{Kon14} and \cite{CRS20}. For the spin case we slightly extend that notion to what we call ``Nakayama-local'' modules, which satisfy the locality condition only up to precomposition of the module action by a Nakayama automorphism (Definition~\ref{def:Nakayama-local}). We arrive at our category $\ALineCat$ of line defects by equivariantising with respect to that action. In the case of $A=1\oplus f$ from above the resulting category will (under mild technical assumptions) be equivalent to the original category of line defects: $\ALineCat \cong \Cc$ as ribbon categories.

This is different to the setting in~\cite{GK16,HLS19}, where one would only keep those defects transparent to $f$. The additional defects describe singularities in the spin structure and are included in the skein-module based construction in~\cite{BM96}. We may think of them as living in the Ramond sector of the cylinder around line defect.

Another result we generalise from~\cite{BM96} is
\begin{itheorem*}{\textsf{\textbf{\!\ref{thm:spin_refinement}.\,}}}
    For the input $A=1\oplus f$ the gauging construction canonically exhibits
    \begin{align*}
        \Zor(\Sigma; A) &\cong \bigoplus_{\sigma\in\mathrm{Spin}(\Sigma)} \Zspin(\Sigma,\sigma)\\
        \intertext{and}
        \Zor(M) &= 2^{-\mu}\sum_{s\in\mathrm{Spin}(M)} \Zspin(M,s)
    \end{align*}
    for $\Sigma$ a closed surface and $M$ a connected bordism, $\mu$ given by $b_0(\partial_- M)$ if the boundary is non-empty and 1 else, and $\mathrm{Spin}(\Sigma)$ and $\mathrm{Spin}(M)$ appropriate classes of spin structure.
\end{itheorem*}
The full theorem is slightly more complicated to state, but also allows for generic choices of $A$ and non-connected $M$. This shows that the gauging construction (or rather the spin part of it) is invertible. Furthermore, the splitting can help to simplify the calculation of some spin invariants, such as the mapping class group representations induced by the spin TFT. An example of such a calculation is given at the end of Section~\ref{subsec:abelian_cs}.

When we apply the construction to Reshetikhin-Turaev TFTs derived from spin modular categories (i.e.\ categories with a distinguished fermionic object) we recover the spin Witten-Reshetikhin-Turaev invariants of~\cite{BM96,Bel98,Bla03,BBC17} (Proposition~\ref{prop:spin-TFT-produces-MF-invariants}) as well as the state spaces calculated in~\cite{BM96,Bel98,Bla03} (Proposition~\ref{prop:spin-RT-state-spaces}). An interesting aspect is that spin gauging in a unitary RT-TFT will require extending the target of the TFT from $\Vect$ to $\sVect$, as was already noted in \cite{BM96} in the class of examples considered there. On a technical level the reason is that the fermion in a unitary spin modular category will never have trivial self-braiding and therefore will never yield a commutative algebra. From a physics perspective this can be expected as a necessary prerequisite to formulate fermionic commutation relations. Note however that non-unitary theories can give bona fide spin invariants even without needing to extend to $\sVect$.

Finally, we compare our results to the classification of abelian spin Chern-Simons theories of~\cite{BM05}. We show that the classifying data is (up to a choice of a third root related to the deprojectification of the mapping class group action) in one-to-one correspondence with pointed spin modular categories. More precisely,~\cite{BM05} classify the spin theories by triples of a finite abelian group $\mathscr{D}$, an equivalence class of quadratic forms $q:\mathscr{D}\to\Qb/\Zb$, and an element $\sigma\in\Zb_{24}$. Disregarding the choice of a third root mentioned above, we will instead take $\sigma\in\Zb_8$. Pointed spin modular categories on the other hand can be classified by triples consisting of a finite abelian group $G$, a normalised quadratic form $\hat q:G\to\Qb/\Zb$ and a fermion $f\in G$. Taking the groupoids of both kinds of triples induced by group isomorphisms compatible with the additional data, we construct an essentially surjective functor from triples $(G,\hat q,f)$ to triples $(\mathscr{D},[q],\sigma)$ which is two-to-one on morphism spaces (Proposition~\ref{prop:spin-abelian-Chern-Simons}). In particular, there is a one-to-one correspondence between isomorphism classes of both kinds of triples. We further show that the state spaces of the gauged theories coincide with the ones of the Chern-Simons theories and that the mapping class group representations on the (unpunctured) torus are isomorphic.

\medskip

The rest of this paper is organised as follows:

In Section~\ref{sec:spin_geometry} we will quickly review the standard description of spin structures on cell complexes and describe the effect of changing the representing cell complex.

In Section~\ref{sec:frobenius_algebras} we will introduce the algebraic data of commutative $\Delta$-separable Frobenius algebras and Nakayama-local modules.

In Section~\ref{sec:gauging_construction} we will present the main gauging construction and the splitting theorem.

In Section~\ref{sec:applications} we finally apply the construction to oriented Reshetikhin-Turaev TFTs and abelian Chern-Simons theories.

\subsubsection*{Acknowledgements}
We thank Sven Möller for helpful comments.
JG and IR are supported in part by the Collaborative Research Centre CRC 1624 ``Higher structures, moduli spaces and integrability'' - 506632645.
IR is partially supported by the Deutsche Forschungsgemeinschaft (DFG, German Research Foundation) under Germany`s Excellence Strategy - EXC 2121 ``Quantum Universe" - 390833306.

\section{Spin Geometry}\label{sec:spin_geometry}
We will review the theory of spin structures on cell complexes. The model we will employ is the one of piecewise linear CW spaces developed in~\cite{Kir12}, as this is good compromise for our application: On the one hand the PLCW structure gives much of the flexibility of general CW complexes, which makes (co)homological data and classical obstruction theory easily accessible, while on the other hand providing a small set of manipulations that translate between different complexes for the same (piecewise linear) manifold.

\subsection{PLCW Spaces}
We state the important definitions and results for PLCW complexes from~\cite{Kir12}. The basic setting is that we work with embedded polyhedra, i.e.\ spaces for which a triangulation exists.

Let $B^n=[-1,1]^n\subset\Rb^n$ denote the standard (piecewise linear) $n$-ball.
\begin{definition}
    An \emph{$n$-cell} in $\Rb^N$ is a subset of $\Rb^N$ together with a splitting into disjoint union $\varphi(\mathrm{int}(B^n))\sqcup \varphi(\partial B^n)$ for some regular piecewise linear map $\varphi:B^n\to\Rb^N$ called the \emph{characteristic map}.
\end{definition}

\begin{definition}
    A \emph{PLCW complex} $K\subset\Rb^N$ is a collection of $k$-cells for $k\le n$. For $n>0$ it is subject to the following inductively defined conditions:
    \begin{enumerate}
        \item For any two distinct $n$-cells $A$ and $B$ with characteristic maps $\varphi_A$ and $\varphi_B$ we have that $\varphi_A(\mathrm{int}(B^n))\cap\varphi_B(\mathrm{int}(B^n)) = \emptyset$.

        \item For any $n$-cell $C$ with characteristic map $\varphi_C$ the boundary $\varphi_C(\partial B^n)$ is a union of cells.

        \item The collection $K^{n-1}$ of cells of dimension at most $n-1$ is a PLCW-complex.

        \item Let $A$ be an $n$-cell given by a characteristic map $\phi:B^n\to\Rb^N$. Then there exists a PLCW decomposition of $\partial B^n$, such that $\phi|_{\partial B^n}$ is a regular piecewise linear map $\partial B^n\to\partial A$ respecting the cell structure.
    \end{enumerate}

    A PLCW decomposition of a compact polyhedron $X\subset\Rb^N$ is a PLCW complex which yields $X$ as the union of its cells.
\end{definition}
Different decompositions of the same manifold can be related by a set of elementary local moves, called elementary subdivisions and their inverses.
\begin{definition}
    An \emph{elementary subdivision} of an $n$-cell is constructed as follows: Assume that the equator of a cell $\varphi(\partial B^n)$ is a union of cells. Then we replace the cell by the following cells: $\varphi(B^n\cap \Hb_+)$, $\varphi(B^n\cap \Hb_-)$ and $\varphi(B^n\cap \Hb_+\cap \Hb_-)$, where $\Hb_\pm$ is the standard closed upper (lower) half space of $\Rb^n$ (with respect to the $n$-th coordinate).
\end{definition}

\begin{figure}
    \centering
    \includegraphics{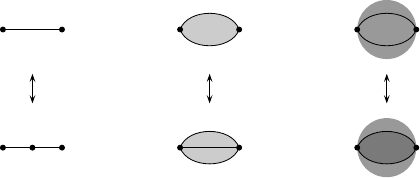}
    \caption{Simple examples of subdivisions of 1-, 2- and 3-cells.}\label{fig:cell_subdivision}
\end{figure}

We can describe this as dividing an $n$-cell into two parts by inserting an $(n-1)$-cell compatible with the rest of the cell decomposition, as shown in Figure~\ref{fig:cell_subdivision}. In this setting we have the following Theorem:
\begin{theorem}[\cite{Kir12}, Thm.\,8.1]
    Let $K$ and $K'$ be two PLCW decompositions of the same compact polyhedron $X$. Then $K$ and $K'$ are related by a finite number of elementary subdivisions and their inverses.
\end{theorem}
This theorem is the primary reason we work with $PLCW$-complexes rather than $CW$-complexes. It allows us to relate all different decompositions by considering only a small number of moves without having to contend with the rigidity and large amount of different moves other decompositions such as triangulation bring.

\subsection{Spin Structures}\label{subsec:spin_structure}
A spin structure on a principal $SO(n)$-bundle $P\xrightarrow[]{SO(n)}M$ is (the isomorphism class of) a double cover of $P$ by a principal $Spin(n)$-bundle $\tilde P\xrightarrow[]{Spin(n)} M$ such that the diagram
\[
    \begin{tikzcd}[row sep=0.1em]
        \tilde P \times Spin(n) \ar[r] \ar[dd] & \tilde P \ar[dd, "2:1"] \ar[dr] & \\
        & & M\ ,\\
        P \ar[r] \times SO(n) & P \ar[ur]
    \end{tikzcd}
\]
with the horizontal arrows given by the actions on the principal bundles, commutes. Alternatively such an isomorphism class of equivariant covers corresponds to a cohomology class $H^1(P;\Zb_2)$ that restricts to the generator $1\in H^1(SO(n);\Zb_2)\cong\Zb_2$ on each fibre. Given an $n$-dimensional vector bundle over an oriented manifold $M$ we can associate to it the bundle of frames, which is a principal $SO(n)$-bundle unique up to homotopy (arising from different choices of Riemannian metric on $TM$). We may thus talk about spin structures for vector bundles. In the case of the vector bundle being $TM$ we speak of a spin structure on $M$. For dimension $n\le2$ it is sometimes useful to talk about the stabilised tangent bundle $\tau_M$ instead. This is natural when we consider these manifolds as living in an (extended) bordism category of higher dimension. In practice this will mean for us that we will consider spin structures on $T\Sigma\oplus\underline{\Rb}$ for surfaces $\Sigma$, where $\underline{\Rb}$ is the trivial line bundle over $\Sigma$, and spin structures on $TM$ in three dimensions.

We will from now on assume that $n\ge 3$.

\begin{figure}
    \centering
    \includegraphics{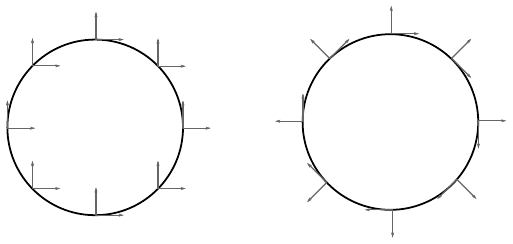}
    \caption{The two spin structures on $S^1$ for a $SO(n)$ bundle with $n\ge3$. We only draw the first two vectors of each frame for visual clarity. For the spin structure on the left the frame is constant and can be extended over the disk by contracting the circle. For the one on the right this is not possible.}
    \label{fig:s1_spin_structures}
\end{figure}
Classical obstruction theory gives a description of spin structures for spaces equipped with a kind of skeleton, such as simplicial complexes and (PL)CW-complexes. Spin structures then are homotopy-classes of trivialisations of the frame bundle over the $1$-skeleton, which must extend over the $2$-skeleton. The latter condition is characterised the vanishing of a characteristic class $w_2\in H^2(M;\Zb_2)$, called the second Stiefel-Whitney class. One can calculate a cocycle representing $w_2$ as follows: For any $2$-cell the trivialisation on the $1$-skeleton determines via pullback a trivialisation of the pullback bundle on $\partial B^2 = S^1$. This trivialisation is either bounding to a trivialisation on $B^2$, in which case $w_2$ is zero, or non-bounding, in which case $w_2$ does not vanish, and so the trivialisation does not describe spin structure. This is illustrated in Figure~\ref{fig:s1_spin_structures}. For details and further generalities on spin structures see~\cite{Mil63}.

We will now describe the interplay of classical obstruction theory with the theory of PLCW spaces in the spin case. When subdividing cells, we only really have to adjust subdivisions of 2-cells: The subdivision of a 1-cell does not affect the trivialisation, and clearly does not affect the extendability over the 2-cells. The subdivision of a 3-cell inserts a new 2-cell, but the added vanishing condition is not independent of those already applied. In particular we think of dividing a 3-ball in half, with a 2-cell attached, say, at the equator of the bounding 2-sphere (cf.\ the third example in Figure~\ref{fig:cell_subdivision}). We know the bundle trivialises over both hemispheres of the 2-skeleton and we can just homotope one of the trivialisations onto the newly inserted 2-cell. Obviously, the subdivision of even higher dimensional cells has no bearing on the spin structure.

For the subdivision of a 2-cell, the homotopy class of the framing along the new 1-cell is uniquely determined by the vanishing condition of the subdivided 2-cell. Both halves of the boundary of the original 2-cell must describe homotopic framed paths for the total path to be bounding. The framing of the new 1-cell must be precisely the one homotopic to both the halves. The vanishing condition for $w_2$ is then satisfied by construction.

\begin{remark}\label{rem:spin-compose}
    For the description of gluing of spin bordisms later on, we will need spin structures on a manifold that extend a fixed trivialisation of the frame bundle at a fixed point. We will always choose this fixed trivialisation to lie over a $0$-cell. On a 1-skeleton these spin structures are then homotopy classes of framings relative to a fixed frame on the distinguished $0$-cell given by the trivialisation of the frame bundle there. The point of looking at these structures is the following: Let $M$ and $M'$ be two manifolds with $\Sigma = \partial_+ M = \partial_- M'$, together with a framed marked point $\{p_i\}$ in each boundary component. Then there exists a canonical spin structure on $M\cup_\Sigma M'$. To see this we use the characterisation of a spin structure on $M$ (or $M'$) as a class in $H^1(Fr_M;\Zb_2)$ that restricts to 1 on the fibres $H^1(SO(n);\Zb_2)\cong\Zb_2$. A spin structure relative to a fixed trivialisation at a point is a class in the relative cohomology $H^1(Fr_M,\{p_i\};\Zb_2)$. From the long exact sequence of the gluing we find:
    \begin{align*}
        0 = H^0(&Fr_M\big|_\Sigma,\{p_i\};\Zb_2) \to H^1(Fr_M\cup_{Fr_M|_\Sigma}Fr'_M,\{p_i\};\Zb_2)\\
        &\to H^1(Fr_M,\{p_i\};\Zb_2)\oplus H^1(Fr'_M,\{p_i\};\Zb_2) \to H^1(Fr_M\big|_\Sigma,\{p_i\};\Zb_2) \to \dots
    \end{align*}
    Assuming the spin structures $s\in H^1(Fr_M,\{p_i\};\Zb_2)$ and $s'\in H^1(Fr'_M,\{p_i\};\Zb_2)$ agree on the boundary, $(s,s')$ is sent to the zero class in $H^1(Fr_M\big|_\Sigma,\{p_i\};\Zb_2)$. They thus come from a unique element of $s\cup_\Sigma s'\in H^1(Fr_M\cup_{Fr_M|_\Sigma}Fr'_M,\{p_i\};\Zb_2)$ which we may take as the glued spin structure by taking the canonical map into $H^1(Fr_M\cup_{Fr_M|_\Sigma}Fr'_M;\Zb_2)$ via the relative long exact sequence.
\end{remark}

\section{Commutative Frobenius Algebras}\label{sec:frobenius_algebras}
The main algebraic input for our gauging construction is that of commutative $\Delta$-separable Frobenius algebra living in a ribbon category $\Cc$. In general we will not make any assumptions as to modularity or semi-simplicity of $\Cc$. The structure of a commutative $\Delta$-separable Frobenius algebra can be seen to naturally arise from trying to read a one-skeleton of a (PL)CW complex as a graph describing the evaluation of an algebraic structure:

Roughly speaking, we need fixed maps from $n$ to $m$ strands at each vertex, where the all strands are arbitrarily divided into incoming and outgoing. This partition should be immaterial, requiring compatible algebra and coalgebra structures. We will further need to choose an ordering of the strands to apply an algebraic operation, and invariance under this ordering leads to (co)commutativity of the (co)algebra structure. The remaining relations arise from manipulations of 1-skeleta that do not change the topology.

We will first state some generalities on Frobenius algebras and their modules and later specialise to the commutative $\Delta$-separable case. To avoid changing the setting too often, let us for simplicity assume from the outset:
\begin{itemize}
    \item[$k$:] a field with $\mathrm{char}(k) \neq 2$,
    \item[$\Cc$:] $k$-linear additive idempotent-complete ribbon category, strict as a monoidal category, with bilinear tensor product, and with absolutely simple tensor unit (i.e.\ $\mathrm{End}_{\Cc}(1) = k \id_1$).
\end{itemize}
For objects $X,Y\in\Cc$ we will denote the braiding by $c_{X,Y}$ and the ribbon twist by $\theta_X$.

\subsection{Algebras and Modules}
    For the graphical notation in ribbon categories we use the conventions of \cite[Sec.\,2.3]{BK01}.

\begin{definition}
    A \emph{(unital associative) algebra} in $\Cc$ is an object $A$ together with morphisms $\mu:A\otimes A\to A$ (multiplication) and $\eta:1\to A$ (unit), such that $\mu\circ(\eta\otimes\id_A) = \mu\circ(\id_A\otimes\eta) = \id_A$ and $\mu\circ(\id_A\otimes\mu)=\mu\circ(\mu\otimes\id_A)$.

    A \emph{coalgebra} is an object $A\in\Obj\Cc$ equipped with maps $\Delta:A\to A\otimes A$ (comultiplication) and $\epsilon:A\to1$ (counit) satisfying dual relations.
\end{definition}

We will often present calculations in terms of string diagrams. We will read these diagrams in composition order, meaning that maps go from the bottom to the top. The structure morphisms of the algebras will be written as:
\[
    \includegraphics[valign=c]{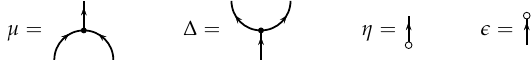}\ .
\]

\begin{definition}
    A \emph{Frobenius algebra} is a quintuple $(A,\mu,\eta,\Delta,\epsilon)$ which defines both an algebra and a coalgebra structure on $A$ which in addition satisfies:
    \[
        \includegraphics[valign=c]{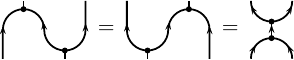}\ .
    \]
    We call a Frobenius algebra
    \begin{itemize}
        \item \emph{commutative} if $\mu\circ c_{A,A} = \mu$,
        \item \emph{$\Delta$-separable} if $\mu\circ\Delta=\id_A$,
        \item \emph{special} if $\epsilon\circ\eta\in k^\times\cdot\id_1$ and $\mu\circ\Delta\in k^\times\cdot\id_A$,
        \item \emph{haploid} if $\dim \hom(1,A) = 1$.
    \end{itemize}
\end{definition}

\begin{definition}
    We have a distinguished Frobenius algebra automorphism $N_A$ called the \emph{Nakayama automorphism} which we will graphically display as a box. Its defining relation is
    \[
        \includegraphics[valign=c]{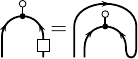}\ .
    \]
    A Frobenius algebra with trivial Nakayama $N_A=\id_A$ is called \emph{symmetric}.
\end{definition}

When we want to display a power of the Nakayama, we will write the exponent into the box.

We can give an explicit formula for the Nakayama automorphism, its inverse and its dual:
\begin{equation}\label{eq:nakayama_formula}
    \includegraphics[valign=c]{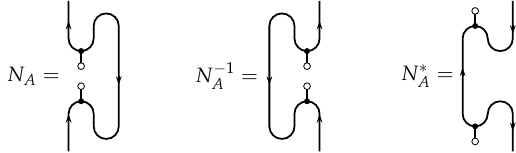}\ .
\end{equation}
See e.g.~\cite{FS08} for details (we use the conventions in \cite{RSW23}).

\begin{definition}
    A \emph{(left) module} over an algebra $A$ in $\Cc$ is an object $M$ together with a map $\rho_M:A\otimes M\to M$ satisfying:
    \begin{itemize}
        \item $\rho_M\circ(\eta\otimes\id_M) = \id_M$.
        \item $\rho_M\circ(\id_A\otimes\rho_M) = \rho_M\circ(\mu\otimes\id_M)$.
    \end{itemize}
    A \emph{morphism of modules} is a map $f:M\to M$ such that $f\circ \rho_M = \rho_M\circ(\id_A\otimes f)$.

    We denote the category of modules over $A$ in $\Cc$ by $\Cc_A$.
\end{definition}
Right modules and bimodules over $A$ are defined analogously.

Let $A$ be a $\Delta$-separable Frobenius algebra in $\Cc$. Then the left modules over $A$ form a monoidal category themselves with the monoidal structure defined by
\begin{equation}\label{eq:module_projector}
    M_1\otimes_A M_2 = \im\includegraphics[valign=c]{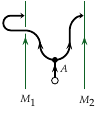}\ .
\end{equation}
Here the action morphism $\rho_M$ is denoted by an arrow head pointing at the module it acts on. Note that the image exists as the above morphism is an idempotent, and by assumption idempotents split in $\Cc$.

\begin{remark}
    In general we have the choice taking the over- or underbraiding in the definition of the projector for tensor product. The choice is ultimately of no importance but must be kept consistent. Throughout this paper we will always take the convention of braiding the algebra over as in \eqref{eq:module_projector}.
\end{remark}

\medskip

The $A$ action on the product is given by
\[
    \includegraphics[valign=c]{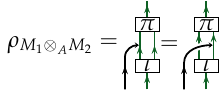}\ .
\]
where $\iota$ denotes the inclusion of $M_1\otimes_A M_1\hookrightarrow M_1\otimes M_2$ in $\Cc$ and $\pi$ the corresponding projector arising from splitting of the idempotent defining $-\otimes_A-$.

\begin{definition}
    Let $A$ be a commutative algebra. We call a module $M\in\Cc_A$ \emph{local} if it satisfies
    \[
        \includegraphics[valign=c]{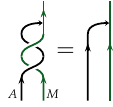}\ .
    \]
    We denote the full subcategory of $\Cc_A$ of local modules by $\Cc_A^\textsf{loc}$.
\end{definition}

    For a special commutative Frobenius algebra, $\Cc_A^\textsf{loc}$ inherits the braiding from $\Cc$, while $\Cc_A$ is in general not braided. The general theory of modules over algebras in a monoidal category is developed e.g.\ in~\cite{KO02}. More details relevant to the later parts of this paper can be found in~\cite{FFRS06}. In Definition~\ref{def:Nakayama-local} we will introduce a slight generalisation of the concept of local modules.

\subsection{Commutative Frobenius Algebras in Ribbon Categories}\label{sec:comm-FA}
Recall that the twist and braiding in a ribbon category are related by
\[
    \theta_{X \otimes Y} = c_{Y,X} \circ c_{X,Y} \circ (\theta_X \otimes \theta_Y) \ .
\]

\begin{proposition}\label{prop:twist_is_nakayama+trivialisation+cocomm}
Let $A$ be a commutative Frobenius algebra $A$ in $\Cc$. Then:
\begin{enumerate}
    \item $N_A=\theta_A$,
    \item $N_A^2 = \id_A$.
    \item $A$ is cocommutative.
\end{enumerate}
\end{proposition}
\begin{proof}
    1) Starting from the defining relation of the Nakayama automorphism, we compute
    \[
        \includegraphics[valign=c]{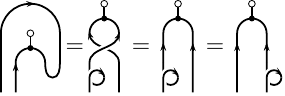}\ .
    \]

    2) By repeating the above calculation pulling the upper strand under (rather than over) the diagram we find that also $N_A=\theta_A^{-1}=N_A^{-1}$.

    3) We calculate
    \[
        \includegraphics[valign=c]{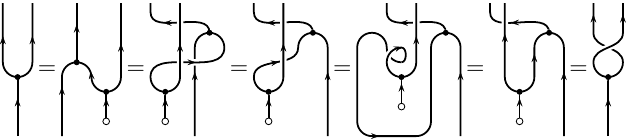}\ ,
    \]
    where we inserted the formula for the inverse Nakayama from \eqref{eq:nakayama_formula} in place of the twist for the second to last equality, which we are allowed to do by 1) and 2) above.
\end{proof}

\begin{remark}
    The relation $\theta_A^2=\id_A$ mirrors the fact that on a given framed line there are two homotopy classes (with fixed endpoints) of the framing. For the case of gauging a bosonic line defect one demands that $A$ be symmetric, i.e.\ that $\theta_A=\id_A$, meaning that we cannot detect different framings of lines (and therefore spin structures). This is in line with the description of an orientation as a trivialisation of the tangent bundle on the 0-skeleton, which can be extended over the 1-skeleton (how exactly being unimportant).
\end{remark}

\medskip

A simple (but useful) relation we will refer to frequently is
\begin{equation}\label{eq:ribbon_direction_change}
    \includegraphics[valign=c]{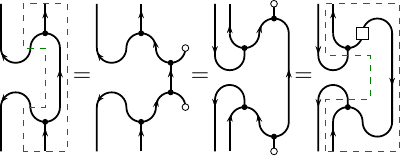}\ .
\end{equation}
While we write it as an equation of string diagrams here, we will most frequently use it to replace part of a diagram of the shape of one of the dashed green boxes by the other box.

In the following we want to fix some notation. For each direct summand $U$ of $A$ we have a restriction $\pi_U: A\to U$ and an inclusion $\iota_U:U\to A$. These induce maps $\mu_{UV}^W:=\pi_W\circ\mu\circ(\iota_U\otimes\iota_V)$ (and $\Delta_U^{VW}$) for direct summands $U,V,W$. In addition we have idempotents $p_U:=\iota_U\circ\pi_U$ projecting onto a component of $A$.

\begin{lemma}\label{lem:F_square_is_1}
    Let $A=1\oplus F$ be a commutative special Frobenius algebra in $\Cc$. Assume additionally that $F \neq 0$ and $\theta_F=-1$. Then $F\otimes F\simeq 1$.
\end{lemma}
\begin{proof}
    First note that $\mu_{FF}^F=-\mu_{FF}^F=0$, as can be seen by inserting a twist before and after $\mu_{FF}^F$ and applying commutativity. The same reasoning gives $\mu_{1F}^1=0=\mu_{F1}^1$.

    Now consider the map $e:=\Delta_A^{FF}\circ\mu_{FF}^A:F\otimes F\to F\otimes F$. By the Frobenius relation it is given by $(\mu_{FA}^F\otimes\id_F)\circ(\id_F\otimes\Delta^{AF}_F)$. Since $\mu_{FF}^F=0$ we see that $e = \Delta_1^{FF}\circ\mu_{FF}^1$, and that we may insert a projector $p_1$ on $A$ between $\mu_{FA}^F$ and $\Delta^{AF}_F$.

    The Hom-spaces $1 \to A$ and $A \to 1$ are one-dimensional, since every morphism $f : 1 \to F$ is zero (naturality of the twist gives $f=-f$), and since by assumption on $\Cc$ we have $\mathrm{End}_{\Cc}(1) = k \id_1$. Thus $\eta = c \iota_1$ and $\epsilon = c' \pi_1$ for some $c,c' \in k^\times$. In particular,  $\eta\circ\epsilon = cc' p_1 : A\to A$. Using this, we get $e= (cc')^{-1} \cdot (\mu_{FA}^F\otimes\id_F)\circ(\id_F\otimes(\eta\circ\epsilon)\otimes\id_F)\circ(\id_F\otimes\Delta^{AF}_F)=(cc')^{-1}\cdot \id_{F\otimes F}$.

    We can also conclude that $\mu_{FF}^1\circ\Delta_1^{FF}$ is non-zero: if it were zero then we would find $\mu_{FF}^1\circ e = \mu_{FF}^1\circ\Delta_1^{FF}\circ\mu_{FF}^1=0$. But since $e$ is known to be $(cc')^{-1}\cdot\id_{F\otimes F}$ we would find $\mu_{FF}^1=0$ which would contradict the non-degeneracy of the pairing $\epsilon\circ\mu$.

    Altogether we have shown that $\mu_{FF}^1\circ\Delta_1^{FF}$ is a non-zero multiple of $\id_1$, and that $\Delta_1^{FF}\circ\mu_{FF}^1$ is a non-zero multiple of $\id_{F \otimes F}$. Hence $F\otimes F\simeq1$.
\end{proof}

\begin{remark}
    In many common settings the above also implies that $F$ is simple.
    It suffices, e.g., that $\Cc$ is a finite tensor category and
    $k$ is algebraically closed of characteristic zero,
    in which case $F$ is simple as any invertible object must have Frobenius-Perron dimension 1 which precludes it from having a non-trivial subobject.
\end{remark}

\begin{theorem}\label{thm:simple_commutative_frobenius_algebras}
    Let $A$ be a haploid commutative $\Delta$-separable special Frobenius algebra in $\Cc$. Then:
    \begin{enumerate}
        \item If $A$ is symmetric then $\dim A = \epsilon\circ\eta \neq 0$.
        \item If $A$ is not symmetric then $\dim A = 0$.
    \end{enumerate}
    When $A$ is not symmetric we further have that $A$ splits as $A=B\oplus F$, where $B$ is a haploid symmetric commutative special Frobenius algebra and $F$ is a local module over $B$ with twist $-1$ which satisfies $F\otimes_B F \cong B$.
\end{theorem}
\begin{proof}
    The first part follows immediately from \eqref{eq:ribbon_direction_change}. Below we therefore assume that $A$ is not symmetric.

    We define maps $p_\pm=\frac{1}{2}(\id_A\pm \theta_A)$ which are idempotents due to Proposition~\ref{prop:twist_is_nakayama+trivialisation+cocomm}. By idempotent completeness of $\Cc$ we can set $B:=\im p_+$ and $F:=\im p_-$. By $p_+ + p_- = \id_A$ we have $B\oplus F=A$, and since $A$ is not symmetric, $p_- \neq 0$ and hence $F \neq 0$. By $\theta_A \circ p_\pm =\pm p_\pm$ we find that the $B$ has trivial twist and $F$ has twist $-1$.

    We can then decompose $\mu$ into the parts $\mu_{IJ}^K$ with $I,J,K\in\{B,F\}$. We find that only $\mu_{BB}^B, \mu_{FF}^B, \mu_{BF}^F, \mu_{FB}^F$ may be non-trivial due to the commutativity of $A$ and the properties of the twist.
    The part $\mu_{BB}^B$ makes $B$ a haploid symmetric commutative Frobenius algebra, and the maps $\mu_{BF}^F$ and $\mu_{FB}^F$ make $F$ a $B$-$B$-bimodule and also a local $B$-module.

    By applying \eqref{eq:ribbon_direction_change} and $\Delta$-separability to $\epsilon\circ\mu\circ\Delta\circ\eta$ one finds
    \begin{equation}\label{eq:dimB-dimF}
        \dim B-\dim F = \epsilon\circ\eta\neq 0 \ .
    \end{equation}
    and we have that at least one of $\dim B$ and $\dim F$ must be non-zero.

    Suppose $\dim F \neq 0$: We use $\Delta$-separability and commutativity of $A$ to find that $\dim F = \mathrm{tr}(\mu_{BF}^F\circ\Delta_F^{BF} + \mu_{FB}^F\circ\Delta_F^{FB}) = 2\cdot\mathrm{tr}(\mu_{BF}^F\circ\Delta_F^{BF})$. We can deform this using \eqref{eq:ribbon_direction_change} and the Frobenius property into a $B$ loop attached to an $F$ loop via a (co)multiplication:
    \begin{equation}\label{eq:bulb_to_dimension}
        \includegraphics[valign=c]{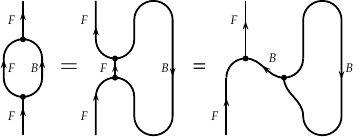}\ .
    \end{equation}
    As $B$ is haploid the $B$ loop on the right hand side evaluates to $\frac{\dim B}{\epsilon\circ\eta} \cdot \eta$. Overall, this results in the equality $\dim F = 2\frac{\dim B \dim F}{\epsilon\circ\eta}$. Dividing by $\dim F$ gives $\dim B = \frac{1}{2}\epsilon\circ\eta\neq0$. Substituting into \eqref{eq:dimB-dimF} shows that $\dim F = - \frac{1}{2}\epsilon\circ\eta$, and so $\dim A=0$.

    Repeating the computation in \eqref{eq:bulb_to_dimension} with all legs coloured by $B$ gives
    \begin{equation*}
        \mu^B_{BB} \circ \Delta^{BB}_B = \frac{\dim B}{\epsilon\circ\eta} \cdot \id_B \ .
    \end{equation*}
    As $\dim B \neq 0$ (and anyway $\epsilon \circ \eta \neq 0$), $B$ is special.

    Suppose $\dim F = 0$: We already saw that in this case $\dim B \neq 0$, and by the above identity, $B$ is special.

    Therefore, in either case $B$ is a haploid commutative special Frobenius algebra in $\Cc$. We now pass to the category $\Cc_B^\mathrm{loc}$ of local $B$-modules in $\Cc$. Since the relative tensor product $\otimes_B$ is described by images of idempotents and $\Cc$ is idempotent-split by assumption, $\Cc_B^\mathrm{loc}$ is again monoidal. The tensor unit $B$ in $\Cc_B^\mathrm{loc}$ satisfies $\mathrm{Hom}_B(B,B) = k \id_B$ as $B$ is haploid, see \cite[Lem.\,4.5]{FS03}. Furthermore, $\Cc_B^\mathrm{loc}$ is a ribbon category \cite[Thm.\,1.17]{KO02}. Thus $A = B \oplus F$ satisfies the conditions of Lemma~\ref{lem:F_square_is_1} in the category $\Cc_B^\mathrm{loc}$. We conclude that $F \otimes_B F \cong B$.

    As dimensions are multiplicative, it follows that $(\dim_B F)^2 = \dim_B B = 1$, where $\dim_B$ denotes the dimension in $\Cc_B^\mathrm{loc}$. If we rescale the coproduct and counit such that $B$ is $\Delta$-separable, by the same computation as in \eqref{eq:bulb_to_dimension}, $\dim_B F = \dim F/\epsilon \circ \eta$. In particular, the case $\dim F = 0$ cannot actually happen.

    This completes the proof.
\end{proof}

\begin{corollary}\label{cor:even_loop_vanishes}
    Let $A$ be a non-symmetric haploid commutative $\Delta$-separable special Frobenius algebra in $\Cc$. Then $\mu\circ(N_A \otimes \id_A)\circ\Delta=0$.
\end{corollary}
\begin{proof}
    Due to a calculation entirely analogous to \eqref{eq:bulb_to_dimension} in the previous proof, but with the full $A$ on every strand, this evaluates to $\frac{\dim A}{\epsilon\circ\eta}\id_A$. By the previous theorem, this is 0.
\end{proof}

\begin{remark}\label{rem:symmetric_subalgebra_rescaling}
    By defining $\mu_B=\mu_{BB}^B$, $\eta_B=\pi_B\circ\eta$, $\Delta_B = 2\Delta_B^{BB}$ and $\epsilon_B = \frac{1}{2}\epsilon\circ\iota_B$ we can give $B$ the structure of a haploid commutative symmetric $\Delta$-separable special Frobenius algebra. Furthermore, $A = B \oplus F$, with the original normalisation of coproduct and counit, is a haploid commutative $\Delta$-separable special Frobenius algebra in $\Cc_B^\textsf{loc}$.
\end{remark}

\subsection{Nakayama-Local Modules}

We will next define a fitting category of modules for commutative Frobenius algebras in ribbon categories. The first ingredient is a suitable locality condition:
\begin{definition}\label{def:Nakayama-local}
    We call a module over a commutative Frobenius algebra \emph{Nakayama-local} if there exists a $\nu\in\Zb_2$ such that
    \[
        \includegraphics[valign=c]{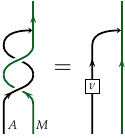}\ .
    \]
    We call the minimal $\nu$ the degree of the module and further call a module even when $\nu=0$ and odd when $\nu=1$. We denote the full subcategory of $\Cc_A$ given by direct sums of Nakayama-local modules by $\Cc_A^\textsf{N-loc}$.
\end{definition}

\begin{remark}\label{rem:nlocal_over_a_implies_local_over_b}
\begin{enumerate}[leftmargin=*]
    \item
        If the algebra is symmetric, then a Nakayama-local module is the same as a local module and the minimal $\nu$ is always 0.

    \item
        In the case that $A$ is haploid $\Delta$-separable special but not symmetric, the twisted action is distinct from the untwisted action on all non-zero modules.
        To see this, consider the morphism $f := \rho_M \circ (\id_A \otimes X) \circ ((\Delta \circ \eta) \otimes \id_M)$, where $X$ stands for the right hand side in the defining equation in Definition~\ref{def:Nakayama-local}. Rewriting $f$ and using Corollary~\ref{cor:even_loop_vanishes} gives
        \[
            f = \rho_M \circ ((\mu\circ(\id_A\otimes N_A^\nu)\circ\Delta\circ\eta)\otimes\id_M) = \begin{cases}0, & \nu=1,\\ \id_M, & \nu=0. \end{cases}
        \]
        Therefore the defining equation for the degree can only be satisfied for either $\nu=0$ or $\nu=1$, but not both.

    \item
        The Nakayama-local modules over $A$ are local modules over the symmetric subalgebra $B$. In fact, by interpreting $A$ as a module over $B$, we find an isomorphism of categories $(\Cc_B^\textsf{loc})_A^\textsf{N-loc}\cong\Cc_A^\textsf{N-loc}$ given by the functor that acts as $((M,\rho_M^B),\rho_M^A)\mapsto (M,\rho_M^A)$ on objects and as the identity on morphisms. An explicit inverse can be given by defining $\rho_M^B:=\rho_M^A\circ\iota_B$.
\end{enumerate}
\end{remark}

From here on we fix:
\begin{itemize}
    \item[$A$:] a commutative $\Delta$-separable special Frobenius algebra in $\Cc$.
\end{itemize}
We now show that $\Cc_A^\textsf{N-loc}$ naturally has the structure of a $\Zb_2$-crossed braided tensor category with respect to the $\Zb_2$-action that twists the action morphism.

Namely, denote by $T$ the automorphism twisting the module action by a Nakayama. The fact that $T$ is monoidal follows from the fact that an insertion of the same number of Nakayamas on both arms of the projector from \eqref{eq:module_projector} does not change the projector. The map $\tilde c_{M,N} := \pi_{T^{\deg M}(N)\otimes_A M} \circ c_{M,N} \circ \iota_{M\otimes_A N}$, where $c_{M,N}$ is the braiding of the underlying objects in $\Cc$, is a twisted braiding. To see that one has to check compatibility with the action. In particular one has to check that moving the action from one strand to the other does not cause a contradiction: When we use $\pi$ to do so we twist the action by $N_A^{\deg M}$, whereas moving the action to the other strand using $\iota$ does not. Therefore calculate
\[
    \includegraphics[valign=c]{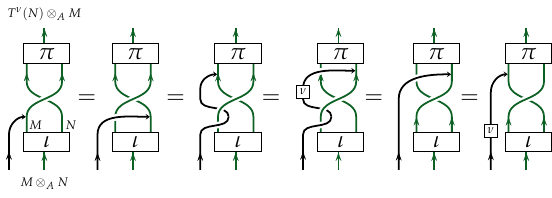}\ ,
\]
where $\nu$ is the degree of $M$. The coherence conditions on $T$ for $\Cc_A^\textsf{N-loc}$ to be $G$-crossed braided given in~\cite[Def.\,8.24.1]{EGNO15} are easy to verify. See loc.\ cit.\ for details on $G$-crossed categories.

Note that $\Cc_A^\textsf{N-loc}$ is -- in general -- not canonically a G-crossed ribbon category. The twist should send an odd module to the corresponding twisted odd module and an even module to itself, i.e. we should have $\theta_M:M\to T^{\deg M}(M)$, but the twist from $\Cc$ will instead do the opposite, that is we have $\theta_M:M\to T^{1+\deg M}(M)$, as can be seen by the ribbon relation:
\[
    \theta_M\circ\rho_M = \rho_M \circ c_{M,A} \circ c_{A,M} \circ (\theta_A\otimes\theta_M) = \rho_M \circ (\theta_A^{1+\deg M}\otimes\theta_M)\,.
\]
To fix this discrepancy we need to choose maps equivariantising the twisted action.

\begin{definition}\label{def:LA-spin}
    We denote the equivariantisation $(\Cc_A^\textsf{N-loc})^{\Zb_2}$ by $\ALineCat$ and call it the category of \emph{spin line defects over $A$}.
\end{definition}

Explicitly the objects of $\ALineCat$ are given by triples $(M,\rho_M,t_M)$, with $(M,\rho_M)$ a left module over $A$ in $\Cc$ and $t_M:M\to M$ the equivariantisation datum for twisting the action morphism by the Nakayama. The equivariance reads as
\begin{equation}\label{eq:module_nakayama_equivariance}
    \includegraphics[valign=c]{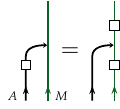}\quad ,
\end{equation}
where we denote an insertion of $t_M$ by a box in the same way we do for the Nakayama. Writing the condition as we did above, $(t_M)^2 = \id_M$ is automatic. We can understand $t_M$ as a pendant to the Nakayama automorphism on the algebra.

We can give the structure maps as
\begin{align*}
    \hat c_{M,N} &= \pi_{N\otimes_A M} \circ (t_N^{\deg M}\otimes \id_M) \circ c_{M,N}\circ\iota_{M\otimes_A N},\\
    \widehat{\mathrm{coev}}_M &= \pi_{M\otimes_A M^*}\circ(\rho_M\otimes\id_{M^*})\circ(\id_A\otimes\mathrm{coev}_M),\\
    \widehat{\mathrm{ev}}_M &= (\id_A\otimes\mathrm{ev}_M\circ(\id_{M^*}\otimes\rho_M)\circ(c_{A,M^*}\otimes\id_M))\circ (\Delta\circ\eta\otimes\iota_{M^*\otimes_A M}).
\end{align*}
In the graphical calculus (of $\Cc$-coloured ribbons) these can be presented as
\[
    \includegraphics[valign=c]{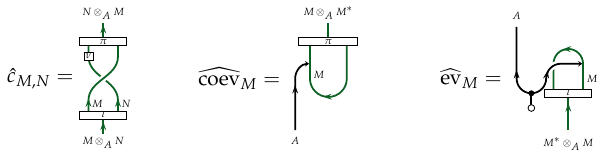}\ ,
\]
where $\nu$ is the degree of $M$. We can additionally define
\[
    \hat\theta_M = t_M^{1+\deg M}\circ \theta_M\,.
\]
The corresponding (co)evaluations of the opposite handedness must then be given by
\[
    \includegraphics[valign=c]{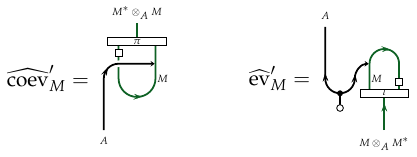}\ .
\]
We see that these come with additional insertions of the equivariantising automorphisms $t_M$.

\begin{proposition}\label{prop:line_defects_are_ribbon}
Via the maps above, $\ALineCat$ is a ribbon category.
\end{proposition}
\begin{proof}
    It is well-known that the equivariantisation of a $G$-crossed braided tensor category is again a braided tensor category, see e.g.~\cite{DGNO10}. It remains to check that $\ALineCat$ is ribbon. Take an equivariant object $(M,\rho_M, t_M)$ of pure degree. Then $\hat\theta_M$ is indeed a morphism in the equivariantisation $\ALineCat$: it commutes with $t_M$ and is an intertwiner for the $A$-action,
    \begin{align*}
        \hat\theta_M\circ\rho_M &= t_M^{1+\deg M} \circ \theta_M \circ \rho_M\\
        &= t_M^{1+\deg M} \circ \rho_M \circ c_{M,A} \circ c_{A,M} \circ (\theta_{A}\otimes \theta_{M})\\
        &= t_M^{1+\deg M} \circ \rho_M \circ (N_A^{\deg M} \otimes \id_M) \circ (N_A\otimes \theta_{M})\\
        &= \rho_M \circ (N_A^{1+\deg M} \otimes t_M^{1+\deg M}) \circ (N_A^{1+\deg M} \otimes \theta_{M})\\
        &= \rho_M\circ (\id_A\otimes\hat\theta_M) \ .
    \end{align*}
Furthermore, $\hat\theta_M$ is a balancing, since
    \begin{align*}
        &(\hat\theta_M\otimes_A\hat\theta_N) \circ \hat c_{N,M} \circ \hat c_{M,N}\\
        &= \pi_{M\otimes_A N} \circ (t_M^{1+\deg M}\otimes t_N^{1+\deg N})
                              \circ (\theta_M\otimes\theta_N)
                              \circ (t_M^{\deg N}\otimes t_N^{\deg M})
                              \circ c_{N,M}\circ c_{M,N} \circ \iota_{M\otimes_A N}\\
        &= \pi_{M\otimes_A N} \circ (t_M^{1+\deg M +\deg N}\otimes t_N^{1+\deg N+\deg M})
                              \circ (\theta_M\otimes\theta_N)\circ c_{N,M}
                              \circ c_{M,N}
                              \circ \iota_{M\otimes_A N}\\
        &= \pi_{M\otimes_A N} \circ t_{M\otimes N}^{1+\deg (M\otimes N)}
                              \circ \theta_{M\otimes N}
                              \circ \iota_{M\otimes_A N}
                              = \hat\theta_{M\otimes_A N} \ .
    \end{align*}
Finally, as $t_{M^*}=t_M^*$ we conclude that $\hat\theta_M$ is indeed a twist.
\end{proof}

Assuming that $A$ is in addition haploid, we can calculate the dimension of an object $M\in\ALineCat$ using the pivotal structure induced by the ribbon structure. One can check that the dimension $\dim_{\ALineCat} M$ is given by the following diagram (in $\Cc$):
\begin{equation}\label{eq:dim-in-Lspin}
    \includegraphics[valign=c]{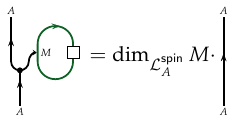}\ .
\end{equation}

Note in particular that we do not obtain the naive module trace of $M$ but rather the twisted one, cf.\ \eqref{eq:dimB-dimF},

\begin{equation*}
    \dim_{\ALineCat} M = \frac{\mathrm{tr}_\Cc(t_M)}{\epsilon\circ\eta} = \frac{\mathrm{tr}_\Cc(t_M)}{\mathrm{tr}_\Cc(t_A)} \ .
\end{equation*}

\begin{proposition}\label{prop:line_defects_for_1f}
    For a commutative Frobenius algebra $A=1\oplus f$ with $f$ a fermion in $\Cc$, that is an object with $\theta_f=-1$ and $f\otimes f\simeq 1$, the induction functor
    \begin{align*}
        U:\Cc\to\ALineCat,\quad x & \mapsto (A\otimes x, \mu\otimes\id_x, N_A\otimes\id_x),\\
                                g & \mapsto \id_A\otimes g
    \end{align*}
    is a ribbon equivalence.
\end{proposition}

\begin{proof}
    Note first that $U$ is indeed a ribbon functor by $U(x)\otimes U(y) = (A\otimes x)\otimes_A (A\otimes y) = A\otimes x\otimes y = U(x\otimes y)$, $U(c_{x,y})=\id_A\otimes c_{x,y} = \hat c_{U(x),U(y)}$ and $U(\theta_x)=\id_A\otimes\theta_x=\hat\theta_{U(x)}$.

    To check that $U$ is an equivalence we give an explicit inverse. For $M \in \ALineCat$ we can decompose $\id_M = p_+ + p_-$ with $p_\pm=\frac{1}{2}(\id_M\pm t_M)$. Define $P_+ : \ALineCat\to\Cc$ as $P_+(M)= \im(p_+)$ and $P_+(h) = h|_{\im(p_+)}$ for $h : M \to N$.

    It is easy to see that $P_+ \circ U \cong \mathrm{Id}_{\Cc}$. Conversely, for $M \in \ALineCat$ let $\iota_+ : P_+(M) \to M$ and $\pi_+ : M \to P_+(M)$ be the embedding and projection maps for $p_+$. Let $\alpha_M : A \otimes P_+(M) \to M$, $\alpha_M = \rho_M \circ (\id_A \otimes \iota_+)$. We will show that $\alpha$ is a natural isomorphism $U \circ P_+ \Rightarrow \mathrm{Id}$.

    First note that $\alpha_M$ is indeed an $A$-module morphism, $\Zb_2$-equivariant, and natural in $M$. To see that each $\alpha_M$ is an isomorphism, one verifies that $\beta_M = 2 (\id_A \otimes \pi_+) \circ (\id_A \otimes \rho_M) \circ ((\Delta \circ \eta) \otimes \id_M) : M \to A \otimes P_+(M)$ is a two-sided inverse. For $\alpha_M \circ \beta_M = \id_M$ one uses Corollary~\ref{cor:even_loop_vanishes}, and for $\beta_M \circ \alpha_M = \id_{A \otimes P_+(M)}$ one first checks that the map $\pi_+ \circ \rho_M \circ (\id_A \otimes \iota_+) : A \otimes P_+(M) \to P_+(M)$ is zero on the subobject $f$ of $A$.
\end{proof}

\begin{corollary}
    By considering the objects of $\Cc_A^\textsf{loc}$ as local $B$-modules we can write $\Cc_A^\textsf{N-loc}\simeq(\Cc_B^\textsf{loc})_A^\textsf{N-loc}$ as explained in Remark~\ref{rem:nlocal_over_a_implies_local_over_b}\,(3). By the previous proposition this induces a ribbon equivalence
    \[
        \ALineCat \simeq \Cc_B^\textsf{loc}\,.
    \]
\end{corollary}

\section{The Gauging Construction}\label{sec:gauging_construction}
In this section we will describe how to gauge a commutative Frobenius algebra in an oriented TFT. We will use the description of spin structures as homotopy classes of framings on a 1-skeleton to obtain a ribbon graph encoding the spin structure. The structure of a commutative Frobenius algebra then algebraically captures the relations between different framings in the same homotopy class and between different skeleta.

We will fix for this chapter:
\begin{enumerate}
    \item[$\Cc:$] an idempotent-complete ribbon category,
    \item[$\TargetCat:$] a symmetric monoidal idempotent-complete category.
\end{enumerate}

\subsection{Spin and Oriented Bordism Categories}\label{subsec:bord_categories}
The setting will be that we work with a fixed three-dimensional oriented TFT, i.e.\ a symmetric monoidal functor
\[
    \Zor:\widehat{\textsf{Bord}}^\textsf{or}_3(\Cc) \to \TargetCat\,,
\]
where
\begin{itemize}
    \item $\textsf{Bord}^\textsf{or}_3(\Cc)$ is the bordism category with closed oriented surfaces with punctures coloured in objects of $\Cc$ as objects and diffeomorphism classes of oriented bordisms with embedded ribbon graphs coloured in $\Cc$ and parametrised boundary as morphisms,
    \item $\widehat{\textsf{Bord}}^\textsf{or}_3\to \textsf{Bord}^\textsf{or}_3$ is a fibration in $\textsf{Cat}$,
    \item $\widehat{\textsf{Bord}}^\textsf{or}_3(\Cc)$ the pullback of $\widehat{\textsf{Bord}}^\textsf{or}_3$ along the forgetful functor $\textsf{Bord}^\textsf{or}_3(\Cc)\to\textsf{Bord}^\textsf{or}_3$.
\end{itemize}

We consider fibrations over $\textsf{Bord}^\textsf{or}_3$ to accommodate TFTs with anomaly such as the Reshetikhin-Turaev construction. The reason that we want the fibration over $\textsf{Bord}^\textsf{or}_3(\Cc)$ to be the pullback of the one over $\textsf{Bord}^\textsf{or}_3$ is that we would like to be able to manipulate the ribbon graph embedded in a bordism without having to worry about affecting the fibre in doing so.

\begin{remark}
     While we do not need the details in the following, let us nonetheless explain why in the example of Reshetikhin-Turaev this is indeed a fibration.

    Recall (e.g.\ from~\cite{LR20}) that a Street fibration $p : \mathcal{E}\to\mathcal{B}$ of a category $\mathcal{E}$ over another category $\mathcal{B}$ is a functor satisfying that for any arrow $f:b\to p e$ in $\mathcal{B}$ there exists an isomorphism $h:b\to pe'$ a $p$-cartesian arrow $g: e'\to e$  in $\mathcal{E}$ such that $f=pg\circ h$. An arrow $g:e'\to e$ in $\mathcal{E}$ is called $p$-cartesian if for any other arrow $g':e''\to e$ in $\mathcal{E}$ and any $h:p e''\to p e'$ in $\mathcal{B}$ such that $pg'=pg\circ h$ there exists a unique lift $k:e''\to e'$ of $h$ in $\mathcal{E}$ such that $g'=g\circ k$.

    The data added to the bordism category in the case of Reshetikhin-Turaev consists of a choice of Lagrangian subspace $\lambda\subset H^1(\Sigma,\mathbb{R})$ for each surface $\Sigma$ and a choice of weight $n\in\Zb$ for each bordism $M$. The weight of the composition of two bordisms is the sum of the weights of the two original bordisms shifted by an additional summand (the Maslov index) computed from the Lagrangian subspaces. We will also denote the resulting bordism category by $\widehat{\textsf{Bord}}^{\textsf{or}}_3$. One can easily convince oneself that every arrow in $\widehat{\textsf{Bord}}^{\textsf{or}}_3$ is cartesian with respect to the forgetful functor to $\textsf{Bord}^{\textsf{or}}_3$: Taking the setup from the definition the bordism $h$ must be endowed with a unique weight $n$ for the condition $g'=g\circ k$ to be fulfilled due to the composition law described above. Thus the forgetful functor $\widehat{\textsf{Bord}}^{\textsf{or}}_3\to\textsf{Bord}^{\textsf{or}}_3$ is a fibration.
\end{remark}

\medskip

For the spin bordism category it is not correct to take plain spin structures on the surfaces and bordisms, since the gluing of two spin structures is in general dependent of the specific choice of lift rather than the isomorphism class. From the perspective of the bordism category one natural way to proceed is to take a fixed spin cover on the entire boundary of a bordism, as is done in~\cite{BM96}. However, an alternative method -- also suggested in~\cite{BM96} -- works better with our more combinatorial approach based on framed skeleta: It is sufficient to choose a specific trivialisation of the (stabilised) tangent bundle at a single point for each boundary component, which can be provided by a choice of frame on that point, and then take the class of all covers that restrict to the ones at the distinguished points. As seen in Section~\ref{subsec:spin_structure} this allows for well-defined gluing. For convenience we will take these frames to be outward pointing for outgoing boundaries and inward for incoming boundaries (meaning that the restriction to the normal bundle of the first two vectors is 0 and the third vector points outward or inward respectively).

We will also refer to such restricted spin structures as spin structures relative to the distinguished framed points $\{p_i\}$. If the $\{p_i\}$ are clear from the context or the choice is irrelevant we will sometimes omit them from the notation.

We will then denote by $\Bordsp$ the bordism category with objects closed surfaces $\Sigma$ equipped with a spin structure relative to the distinguished framed points on $T\Sigma\oplus\underline{\Rb}$ and morphisms bordisms $M$ with parametrised boundary equipped with a spin structure $s$ compatible with the parametrisation also relative to the distinguished framed points.

We next want to lift this to an appropriately fibred bordism category. We take the following pullback diagram in the (strict) 2-category of categories
\[\begin{tikzcd}
    \hBordsp \ar[r] \ar[d] \ar[phantom]{dr}[pos=0.05]{\lrcorner} & \widehat{\textsf{Bord}}^\textsf{or}_3 \ar[d]\\
    \Bordsp \ar[r] & \textsf{Bord}^\textsf{or}_3\ .
\end{tikzcd}\]
The bottom and right arrows are the respective forgetful functors. It can be realised as a pseudopullback due to $\widehat{\textsf{Bord}}^\textsf{or}_3\to\textsf{Bord}^\textsf{or}_3$ being a fibration, as shown in~\cite{JS93}. Since fibrations are stable under pullback we get a fibration $\hBordsp\to\Bordsp$.

\begin{example}
In the case of the anomaly in the Reshetikhin-Turaev TFT, the corresponding pullback to a spin bordism category then simply has triples $(\Sigma,\sigma,\{p_i\},\lambda)$ as objects and $(M,s,n)$ as morphisms. We see that the additional structures on the bordisms category do not interact in this construction.
\end{example}
We may also want to consider external lines in our spin TFT. Those lines will be coloured by elements of $\ALineCat$. The even part of the line defects can -- in the presence of an $A$-action -- be braided through the skeleton without issue. The odd graded part however can not. When trying to work with ``honest'' spin structures one has to project out the odd part. This is the approach taken for example in the gauging construction of~\cite{HLS19}. A different approach, which we will follow, was developed in~\cite{BM96}. There one allows spin structures to become singular around a line defect. This approach is also mentioned in~\cite[Sec.\,4.1.4]{ALW19}.

\begin{definition}
    A \emph{singular spin structure on a 3-manifold $M$} with embedded $\ALineCat$-coloured admissible ribbon graph $E$ is a spin structure on $M\setminus E$. 

    A \emph{singular spin structure on a surface $\Sigma$} with $\ALineCat$-coloured punctures $P$ is a spin structure on the restriction of $T\Sigma\oplus\underline{\Rb}$ to $\Sigma\setminus P$.
\end{definition}

\begin{definition}
    An $\ALineCat$-coloured ribbon graph 
    in a 3-manifold $M$ with singular spin structure
    is called \emph{admissible} when each strand is coloured by either a purely even or purely odd object such that any disk embedded into $M$ is punctured by an even total colour if the boundary circle has bounding framing and an odd total colour when it does not, where total colour means the tensor product over the colours of all punctures. 
    
    A set of $\ALineCat$-coloured punctures on a surface is \emph{admissible} if the colour of each puncture is purely even or purely odd, depending on whether the framing on a circle around the puncture is bounding or not.
\end{definition}

\begin{definition}
    We define the category $\Bordsp(\ALineCat)$ as follows:
    \begin{itemize}
        \item The objects are surfaces $\Sigma$ equipped with one distinguished framed point $p_i$ per connected component, and with framed punctures $P$ coloured admissibly by $\ALineCat$ and a corresponding singular spin structure $\sigma$ which restricts to the framing on each $p_i$.

        \item The morphisms $(M,E,s):(\Sigma,P,\sigma,\{p_i\})\to(\Sigma',P',\sigma',\{p_i'\})$ are (relative) diffeomorphism classes three-manifolds $M$ with incoming parametrised boundary $\Sigma$ and outgoing parametrised boundary $\Sigma'$. $M$ is equipped with an embedded ribbon graph $E$ coloured admissibly in $\ALineCat$ such that $E\cap\Sigma=P$ and $E\cap\Sigma'=P'$, and with a singular spin structure $s$ extending $\sigma$ and $\sigma'$ and restricting to the frames at $\{p_i,p_i'\}$.
    \end{itemize}
\end{definition}

This definition can again be lifted along a fibration of bordism categories. We will from now suppress the extensions of bordism categories given in terms of fibrations since they have no bearing on the construction itself. All constructions presented below should be thought of as taking place in an appropriately extended bordism category.

It is important to our construction to be able to locally manipulate the embedded ribbon graphs, which can be interpreted as a kind of locality condition. Consider an bordism $M$ with an embedded ribbon graph $E$ coloured by some ribbon category $\mathcal{D}$. If we take an embedded three-ball $B=I\times I\times I$ in the interior of $M$ such that $\partial B$ intersects $E$ transversally without intersecting with any coupons and such that $\partial B\cap E \subset I\times\{0\}\times\{0\}$, we get an object in the category of $\mathcal{D}$-coloured ribbon graphs.

To $\mathcal{D}$-coloured ribbon graphs we can assign maps in $\mathcal{D}$ via the Reshetikhin-Turaev functor $\mathcal{F}_\mathcal{D}^\textsf{RT}$, as defined in~\cite[Ch.\,I.2]{Tur10}.

\begin{definition}\label{def:ribbon_tft}
    We say a TFT \emph{respects $\mathcal{D}$-skein relations} when it is invariant under replacements of an embedded three-ball $B$ with any other three-ball $B'$ which has the same boundary and satisfies $\mathcal{F}_\Cc^\textsf{RT}(B)=\mathcal{F}_\Cc^\textsf{RT}(B')$.
\end{definition}
In practice this means we are allowed to apply ribbon relations on the graphs embedded in $M$. We will assume $\Zor$ to respect $\Cc$-skein relations.

\subsection{Framed 1-Skeleta and Ribbon Graphs}
Consider a closed (PLCW) $3$-manifold $M$  together with a trivialisation of the tangent bundle $TM$ on the corresponding 1-skeleton. This gives us a framed graph embedded in $M$. We want to translate this framed graph, or rather its homotopy class, into a ribbon graph, modulo some relations to remove ambiguities introduced by the presentation as a ribbon graph.

We begin by translating the $0$-skeleton, which is simply a collection of (3-)framed points, into coupons. To do so we choose an $3$-ball around each vertex, such that the intersection of its boundary with the 1-skeleton consists of precisely one point for each 1-cell incident on that vertex (or two in the case of a self-loop). A coupon is the embedding of the unit square $I^2$ into that ball. It comes with two preferred tangent vectors, one along each copy of the unit interval, which we take to make up the standard frame $(e_1,e_2)$. We embed the coupon in such a way, that the first two vectors of the original frame at the point agree with the preferred frame in the centre of the coupon. The third vector of the original frame at the 0-cell now gives the normal direction.

For the 1-cells we choose tubular neighbourhoods in $M$ with the balls corresponding to the $0$-cells removed.

We must then translate the framed 1-cells into ribbons. Just as coupons, ribbons, which are embeddings of $I\times I$ or $S^1\times I$, come with two preferred tangential directions and we can generate a preferred frame over the core $I\times\{\frac{1}{2}\}$ in the same way as for the coupons. Then one can complete to a 3-frame by taking the positively oriented normal vector. We can in general not demand the ribbon to induce the same framing as the original 1-cell and instead will demand that the induced framings are homotopic. That this is possible is easy to see: There are only two homotopy classes of framing for the 1-cell, given fixed framings at the endpoints. We may now arbitrarily place a ribbon with its core along the original 1-cell (outside of the neighbourhood of each vertex) and connect it to the coupon corresponding to the vertex somewhere along $\{0\}\times I$ if the ribbon runs into the coupon and along $\{1\}\times I$ if it runs out of the coupon. The resulting homotopy class either already agrees with the framing on the 1-cell or can be made to agree by the insertion of an additional twist.

Note we strictly divide ribbons into incoming and outgoing and attach these to separate sides of the coupons. We do so to make sure that the frames on the whole graph always fit together.

When connecting all of the ribbons to their respective coupons we are met with more choices: We have to choose an ordering among the incoming and outgoing ribbons at each coupon and the related choices of over- and undercrossings. Further we have already chosen a direction for each ribbon. Thinking of the above choices as generators of an equivalence relation, what we have obtained now is a representative of an equivalence class of ribbon graphs coming from each skeleton corresponding to the same spin structure. Explicitly, the equivalence relation is generated as follows:
\begin{enumerate}
    \item We identify all ribbon graphs related by ambient isotopy.
    \item We identify all ways to order the ribbons at vertices. The generating relation for this is shown in Figure~\ref{fig:coupon_relations}.
    \item We identify both choices of direction in the way shown in Figure~\ref{fig:ribbon_rev_relation}.
    \item We may add or remove any number of double twists (or inverse double twists) from each ribbon.
\end{enumerate}
Note that the fourth relation is implied by the first three via an application of the Dirac belt trick. This is the topological analogue of Proposition~\ref{prop:twist_is_nakayama+trivialisation+cocomm}\,(2).

\begin{figure}
    \centering
    \includegraphics[]{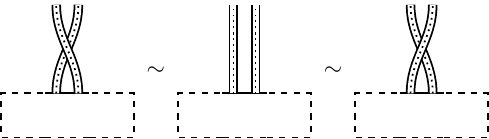}
    \caption{The generators for the set of relations on connections on the coupon. The direction of the ribbons may be either incoming or outgoing}
    \label{fig:coupon_relations}
\end{figure}

\begin{figure}
    \centering
    \includegraphics[]{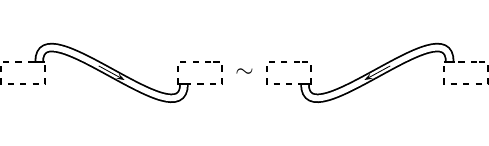}
    \caption{The relation for reversing a ribbon.}
    \label{fig:ribbon_rev_relation}
\end{figure}

We will further extend the discussion to 3-dimensional bordisms of surfaces with a distinguished marked framed point. We will once again take a PLCW decomposition of $M$ and take a trivialisation of the tangent bundle $TM$ on the part of the $1$-skeleton inside the bulk and a trivialisation of $T\Sigma\oplus N\Sigma$ on the part of the $1$-skeleton at the boundary, where $N$ is the normal bundle. In addition, we will demand that the marked points and their framings coincide with one of the 0-cells on the respective boundary components. We now thicken the manifold slightly at the boundaries by gluing on $\Sigma_i\times[0,\epsilon]$ (or $\Sigma_i\times [-\epsilon,0]$, depending on whether the boundary component is incoming or outgoing) for each boundary component $\Sigma_i$. The entire original $1$-skeleton now lives within the bulk and we can proceed to convert it into a ribbon graph as before. For each marked point we add a ribbon from the coupon corresponding to the marked point to $\Sigma\times\{0\}$ with the framing corresponding to the framing of $p_i$ moved constantly along $[0,\epsilon]$ (or $[-\epsilon,0]$, as the case may be). An example is shown in Figure~\ref{fig:s2_extend_to_boundary}.

\begin{figure}
    \centering
    \includegraphics[width=0.6\textwidth]{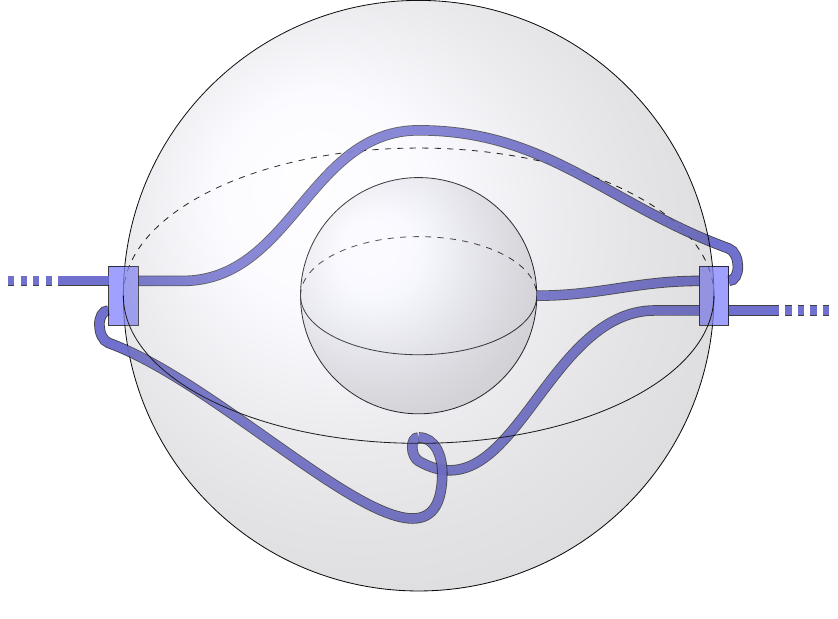}
    \caption{Taking the 2-2-2-stratification of the sphere $S^2$, we can draw an example of the (neighbourhood) of the $S^2\times[0,\epsilon]$ part. The outer sphere goes on to the rest of the bordism, while on the inner sphere, i.e.\ at $S^2\times\{0\}$ we have punctures inserted for the distinguished coupon in $S^2\times\{\epsilon\}$.}
    \label{fig:s2_extend_to_boundary}
\end{figure}

\begin{definition}\label{def:framing_ribbon_class}
    We assign to each framed 1-skeleton of $M$ with framing $fr$ and distinguished boundary points $\{p_i\}$ a class of ribbons graphs $\mathcal{R}\textsf{Frames}(fr;\{p_i\})$ under equivalence relation generated by those of the relations (1)--(4) above that leave the coupon and ribbon corresponding to each $p_i$ in place.
\end{definition}

\subsection{Spin Structures as (Coloured) Ribbon Graphs}
As just described, to the presentation of a spin structure on a 3-manifold $M$ as the homotopy class of a framing of a given 1-skeleton of a PLCW-decomposition of $M$, we can assign the class $\mathcal{R}\textsf{Frames}(fr;\{p_i\})$ of ribbon graphs. There are still multiple classes of ribbon graph corresponding to each spin structure relative to $\{p_i\}$, as we have not yet considered the effect of different choices of 1-skeleton.

We thus add an additional relation to the equivalence class: We identify the images of 1-skeleta achieved by cellular subdivisions as explained in Section~\ref{subsec:spin_structure}. Note that only the subdivisions of 1- and 2-cells directly affect the 1-skeleton. Higher-dimensional subdivisions will not change the 1-skeleton directly, but are nevertheless important as they change the possible subdivisions in dimension 1 and 2 available to us. Note further that it is possible to keep the marked points on the boundary fixed under all these operations. (As shown in~\cite[Thm.\,8.1]{Kir12} one can reach a common subdivision by only applying radial subdivisions, which can be performed without touching the marked point.)

For a given spin structure $s$ relative to $\{p_i\}$ call the class of ribbon graphs we assign to it $\Rc\textsf{Spin}(s;\{p_i\})$. Furthermore, given a ribbon graph $G$ call the associated class of ribbon graphs equivalent under $\Cc$-skein relations $\Rc\textsf{Graphs}(G)$.

\begin{definition}
    Given a $\Delta$-separable Frobenius algebra $A$ denote by $\mu_{k,\ell}$ the unique map $A^{\otimes k}\to A^{\otimes \ell}$, composed only of (co)multiplications such that the representing graph is connected.
\end{definition}

For example, one can write $\mu_{k,\ell}$ as $A^{\otimes k} \xrightarrow{a} A \xrightarrow{b} A^{\otimes \ell}$, where $a$ is a $(k-1)$-fold product, and $b$ a $(\ell-1)$-fold coproduct.

\begin{lemma}\label{lem:loop_removal}
    Let $G$ be a connected ribbon graph in a 3-ball that can be evaluated under $\Fc^\textsf{RT}$ such that the ribbons are all coloured by a $\Delta$-separable commutative Frobenius algebra $A$ in a ribbon category $\Cc$ and the coupons are all coloured by the maps $\mu_{i,j}$. Assume all loops in the graph have bounding framing. Let $R$ be a strand in $G$ not attached to the boundary, such that removing the strand does not disconnect $G$ and denote the graph with $R$ removed $G\setminus R$. Then $\Fc^\textsf{RT}(G) = \Fc^\textsf{RT}(G\setminus R)$.
\end{lemma}
\begin{proof}
    This is in essence a restatement of $\Delta$-separability.

    We will first treat the case of a ribbon going from a coupon to itself. If we cancel twists by applying Proposition~\ref{prop:twist_is_nakayama+trivialisation+cocomm}, we wind up with a loop with exactly one twist on it. To apply $\Delta$-separability we simply insert an identity coupon on the loop. In the planar projection one half of the split ribbon must now have the twist on it. We can invert the attaching lines by \eqref{eq:ribbon_direction_change} which gains us an additional twist, which we may again cancel. In the graphical calculus of ribbon categories this is
    \[
        \includegraphics[valign=c]{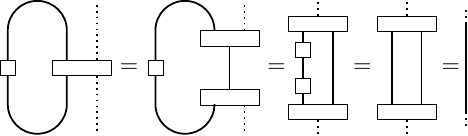}\ .
    \]
    Take now an arbitrary $R$. Commutativity and the Frobenius property allow us to take one end of the ribbon and slide it along the skeleton. Due to the assumption that $G\setminus R$ is connected, we can find a different path in $G$ to the other end $R$. Sliding the ends together in this way, we have almost reduced the situation to that of the single loop above, but we may have linked part of the rest of the graph.
    By commutativity of $A$, in a connected $A$ graph $G$ we have $\Fc^\textsf{RT}(G)= \Fc^\textsf{RT}(G')$, where $G'$ differs from $G$ by exchanging an over-braiding for an under-braiding. In this way, one can unlink the loop without changing the evaluation under $\Fc^\textsf{RT}$.
\end{proof}

\begin{proposition}\label{prop:A_graph_equiv_class}
    If we choose the colour of every ribbon to be a (fixed) commutative, $\Delta$-separable Frobenius algebra $A$ in a ribbon category $\Cc$ and the colour of every coupon with $k$ incoming and $\ell$ outgoing ribbons by $\mu_{k,\ell}$, then for a given spin structure $s$ relative to $\{p_i\}$ and for all $G,G'\in\Rc\textsf{Spin}(s;\{p_i\})$ we have $\Rc\textsf{Graphs}(G)=\Rc\textsf{Graphs}(G')$.
\end{proposition}

\begin{proof}
    The definition of both equivalence classes contain ambient isotopy, so we have to check that the local replacements generate all other relations in the definition of $\Rc\textsf{Spin}$.

    The reordering of ribbons incident on a coupon is equivalent to the commutativity of $A$. The relation reversing a ribbon is \eqref{eq:ribbon_direction_change}.

    The subdivision of a $1$-cell is trivial; the subdivision of a $2$-cell follows from Lemma~\ref{lem:loop_removal}, which reduces it to $\Delta$-separability. We apply the lemma in a regular neighbourhood of the 2-cell(s) in question, which gives us a ball in which we can use $\Cc$-skein relations.
\end{proof}

\subsection{Construction of \texorpdfstring{$\Zspin$}{the Spin TFT}}
Let $A$ be a commutative $\Delta$-separable Frobenius algebra in $\Cc$ and let $\Zor:\widehat{\Bord}_3^\textsf{or}(\Cc)\to\TargetCat$ be an oriented TFT which respects $\Cc$-skein relations. In the following we will disregard the data of the fibre, as it is irrelevant to the construction, and only use the base category $\Bord_3^\textsf{or}(\Cc)$. We will use the encoding of a spin structure $s$ on a manifold $M$ with boundary distinguished points $\{p_i\}$ into a ribbon graph $G_s \in\Rc\textsf{Spin}(s;\{p_i\})$ to insert a spin dependence into an oriented TFT. We have the following corollary to Proposition~\ref{prop:A_graph_equiv_class}:

\begin{corollary}
    The evaluation $\Zor(M,G_s)$ for a 3-manifold $M$ with boundary is only dependent on the spin structure $s$ relative to $\{p_i\}$ and not on the choice of representative $G_s$.
\end{corollary}

Let $(\Sigma,\sigma,p)$ be an object of the spin bordism category with connected, non-empty underlying surface $\Sigma$. Take $\Sigma\times[0,1]$ with the product extension $s_0$ of the spin structure $\sigma$ and define
\[
    \Zspin(\Sigma,\sigma,p) := \im\ \Zor(\Sigma\times I,G_{s_0})\,.
\]
In the case of $\Sigma=\emptyset$, we simply set $\Zspin(\emptyset) = \Zor(\emptyset)$.
Note that $\Zor(\Sigma\times I,G_{s_0})$ is an idempotent and hence that this image always exists due to the idempotent completeness of $\TargetCat$.

\begin{remark}
    There are precisely two spin structures on $\Sigma\times I$ relative to $\{p_i\}$ that restrict to $\sigma$ at the boundary. We can distinguish $s_0$ on the skeleton in the following way: For a given $G_{s}$, choose a path from $(p,0)$ to $(p,1)$ in the skeleton that is homotopic (as a plain path with fixed end points, not as a framed path) to the constant path along $\{p\}\times I$. If the chosen path is also homotopic to $\{p\}\times I$ (with the constant frame along $I$) as a framed path, then $s=s_0$.
\end{remark}

\begin{definition}
    Let $M$ be a $3$-manifold with incoming boundary $(\Sigma,\sigma,\{p_i\})$ and outgoing boundary $(\Sigma',\sigma',\{p_i'\})$ equipped with a spin structure $s$ relative to $\{p_i,p_i'\}$ and let $G_s\in\Rc\textsf{Spin}(s,\{p_i\})$ be an $A$-coloured graph representing $s$ as above. Denote by $A^{\{p_i\}}$ the set of punctures on $\Sigma$ induced by $G_s$. We define
    \begin{align*}
        \Zspin&(M,s) :=\\
        &\Zspin(\Sigma,\sigma,\{p_i\}) \xrightarrow{\iota} \Zor(\Sigma;A^{\{p_i\}})&\xrightarrow{\Zor(M, G_s)}\Zor(\Sigma';A^{\{p_i'\}})\xrightarrow{\pi}\Zspin(\Sigma',\sigma',\{p_i'\}),
    \end{align*}
    where $\iota$ and $\pi$ are the inclusion and projections maps coming from the split idempotent above.
\end{definition}

\begin{theorem}
    The assignments of $\Zspin:\hBordsp\to\TargetCat$ define a symmetric monoidal functor.
\end{theorem}
\begin{proof}
    The main point we have to check is that the assignment is compatible with composition. Take two composable bordisms $(M_1,s_1):(\Sigma,\sigma,\{p_i\})\to(\Sigma',\sigma',\{p_i'\})$ and $(M_2,s_2):(\Sigma',\sigma',\{p_i'\})\to(\Sigma'',\sigma'',\{p_i''\})$. Then the composition $\Zspin(M_2,s_2)\circ\Zspin(M_1,s_1)$ has in the middle the composition
    \[
        \dots\to\Zor(\Sigma',A^{\{p_i'\}})\xrightarrow{\pi'}\Zspin(\Sigma',\sigma',\{p_i'\})\xrightarrow{\iota'}\Zor(\Sigma',A^{\{p_i'\}})\to\dots\ ,
    \]
    where $\iota'\circ\pi'$ gives the idempotent defining the state space. By construction, that idempotent is the evaluation of a morphism under $\Zor$ and we rewrite
    \[
        \Zspin(M_2,s_2)\circ\Zspin(M_1,s_1)=\pi''\circ\Zor(M_2,G_{s_2})\circ\Zor(\Sigma'\times I,G_{s_0})\circ\Zor(M_1,G_{s_1})\circ\iota\,.
    \]
    Choosing framed skeleta to construct $G_{s_0}$ and $G_{s_1}$ that agree on $\Sigma'\subset M_1$ and $\Sigma'\times\{0\}$ it is easy to see that $G_{s_1}\cup_{\{p_i'\}} G_{s_0}$ is equivalent under skein relations (by contracting the two copies of the one-skeleton on the glued boundary) to one representing $M_1\cup_{\Sigma'} (\Sigma'\times I)$ with the spin structure induced on the composition (cf.\ Remark~\ref{rem:spin-compose}). Doing the same for $M_2$ we arrive at
    \begin{align*}
        \Zspin(M_2,s_2)&\circ\Zspin(M_1,s_1)\\
        &=\pi''\circ\Zor(M_2\cup_{\Sigma'}(\Sigma'\times I)\cup_{\Sigma'} M_1,G_{s_2}\cup_{\{p_i'\}}G_{s_0}\cup_{\{p_i'\}}G_{s_1})\circ\iota\\
        &= \Zspin(M_2\circ M_1,s_2\circ s_1).
    \end{align*}
    \phantom{.}
\end{proof}

\subsection{Line Defects in \texorpdfstring{$\Zspin$}{the spin TFT}}
Take a morphism $(M,E,s)\in\Bordsp(\ALineCat)$, where $s$ is a singular spin structure. We insert a ribbon graph $G_s$ into $M$ describing $s$ via a skeleton of $M\setminus N(E)$, where $N(E)$ is a regular neighbourhood of $E$.

Our goal now is to integrate the $\ALineCat$-coloured ribbon graph with the spin skeleton in a way appropriate to both the equivariant structure on $\ALineCat$ and the (singular) spin structure on $M$. Doing so requires some mild alterations to the input graph $E$. We then prove that applying these steps will yield a spin TFT respecting $\ALineCat$-skein relations.

When we are given an arbitrary ribbon graph as input, we may encounter coupons which have incoming and outgoing ribbons incident to the same side. So far we are not equipped to talk about the framing in the associated graph, since there will be a mismatch between the standard frame of the coupon and the one of the ribbons and as we have seen the construction of $G_s$, changing the direction is a non-trivial operation. It is however still useful to allow such coupons, for example to encode (co)evaluations as coupons. Given a general graph we will slightly modify it to isolate the problematic points: We insert an identity-coloured coupon as in Figure~\ref{fig:2io_coupons} for each incorrectly oriented ribbon entering a coupon such that the ribbon now enters the new coupon in the wrong direction but the original coupon correctly. We then assign special ways to transport the framings through these two input or two output ribbons as shown in the figure. Note that these are purely a book-keeping device for when we need the framing of the graph and can be ignored otherwise.

The last modification to the coupons we make is inserting projections and injections to and from the relative tensor product $-\otimes_A-$ before and after the map the coupon is coloured by. This is done to adjust between the different monoidal structures on $\ALineCat$ and $\Cc$. For parts of $E$ not in the same connected component we will later see that this is mediated by connections to between $G_s$ and $E$ acting as a sort of generalised projector, see Proposition~\ref{prop:zspin_is_ribbon} below. Further, maps to the monoidal unit in $\ALineCat$ are maps to $A$ in $\Cc$. To make those maps to the unit in $\Cc$ we compose with $\epsilon:A\to1$. Similarly maps from the monoidal unit may be handled via composition with $\eta:1\to A$.

\begin{figure}
    \centering
    \includegraphics{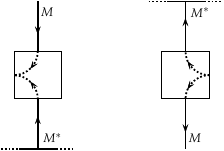}
    \caption{The coupons with two incoming and two outgoing ribbons. Dashed inside we show how the framing is transported across each coupon. The original coupons are indicated by the bottom left and top right bases.}\label{fig:2io_coupons}
\end{figure}

Next we will need to make the connections between $G_s$ and $E$. For each connected component of $E$ not ending on the boundary we make a single connection from somewhere on $G_s$ (the precise location is irrelevant -- we can just move it to where we would like it to be since $G_s$ is connected) to somewhere on the connected component. The framing of that connection is arbitrary since we can just slide any twist on the connection onto $E$ as a pair of $t_M$. We can see that this will not change the evaluation under $\Zor$ as follows: At each coupon of $E$ choose an incident ribbon and insert a $t_M^2=\id_M$. Pushing one of the $t_M$ over the coupon leads to having one $t_M$ insertion for each connection to the coupon (recall from Definition~\ref{def:LA-spin} that coupons of $E$ are compatible with the equivariant structure, i.e.\ they commute with $t_M$'s). We now have inserted two $t_M$ on each ribbon, which can be annihilated, changing the framing of the action in the process due to~\eqref{eq:module_nakayama_equivariance}.

For those components of $E$ that do end on the boundary we add an action at each puncture, again with arbitrary framing, originating from the boundary component of the puncture. For a given boundary component denote a given choice of framed line for the punctures $P$ by $C_P$.

\begin{figure}
    \centering
    \includegraphics[width=0.7\textwidth]{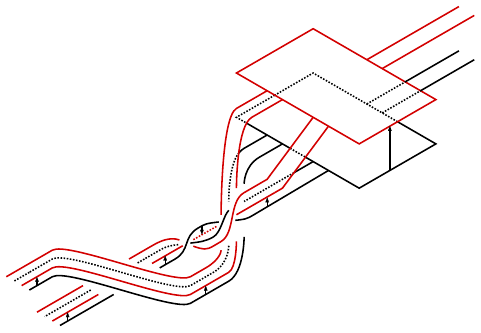}
    \caption{Pushing off part off a ribbon graph from itself. The original graph is shown in black and the push-off in the normal direction is shown in red. The normal direction is indicated by arrows.}
    \label{fig:pushoff}
\end{figure}

In general, the framing of the embedded graph $E$ will disagree with the canonical framing from the spin structure in the sense that, for any push-off of in $E$ into the interior of $M\setminus E$ (the process is illustrated in Figure~\ref{fig:pushoff}), the framing induced by the restriction of spin structure to the push-off
will not be homotopic to the ribbon framing. The difference between the two defines a class in $H^1(G_s\cup\{\text{actions}\}\cup E;\Zb_2)$. We add $t_M$ onto the graph according to that difference. (I.e.\ we add a $t_M$ to each ribbon segment evaluating to 1 under some representative of the class that vanishes on $G_s$ and the actions.)

 We will denote the total resulting graph $G_s(E,C_P,C_P')$, omitting the $C_P$ or $C_P'$ when they are empty, i.e.\ when there are no punctures.

\begin{remark}
    We can view adding a $t_M$ as independently changing the framing of the $A$-part of a module $M=A\otimes x$ to match the spin structure, while leaving the framing on the $x$-part the same. For an induced module we may thus ``absorb'' the $A$ part back into the skeleton.
\end{remark}

\begin{figure}
    \centering
    \includegraphics{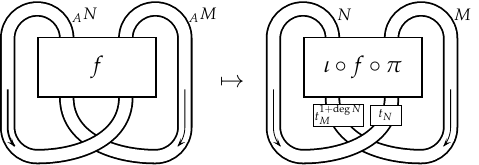}
    \caption{Translation of a simple ribbon graph from $\ALineCat$ to $\Cc$. Here $f:M\otimes_A N\to N\otimes_A M$ is a map in $\ALineCat$ and $\pi$ and $\iota$ are again the projection from and inclusion to the respective tensor products in $\Cc$.}
    \label{fig:t_insertion_example}
\end{figure}

\begin{example}
    Consider the graph in Figure~\ref{fig:t_insertion_example}. The left hand loop will not link to the original graph after pushing off: Pushing the loop off itself in the positive direction (i.e.\ out of the page) is not obstructed by any part of the graph and the loop will thus not link to anything. The right hand loop however will link to the left hand loop: Pushing it off we get stuck under the left hand loop. So we insert a single instance of $t_N$ on the left hand loop.
    On the right loop we may need to insert an additional $t_M$ depending on the degree of the colour of the left hand side.
\end{example}

\begin{example}
    \begin{figure}
        \centering
        \includegraphics[width=0.8\textwidth]{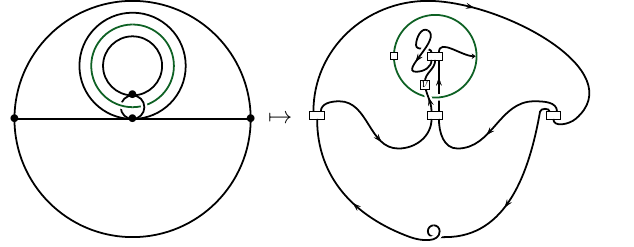}
        \caption{A possible skeleton for a $\ALineCat$-coloured ribbon of degree $\nu$ in $S^3$ and a possible translation into the ribbon graph $G_s(E)$. The module is shown in green and the algebra in black.}
        \label{fig:loop_in_s3}
    \end{figure}
    Consider the embedding of an even $0$-framed external line of degree $\nu$ into a 3-sphere. A possible skeleton and translation into a ribbon graph is shown in Figure~\ref{fig:loop_in_s3}. The ribbon graph can be simplified to find that we just rescale the skeleton by a factor of $\dim_{\ALineCat} M$ as in \eqref{eq:dim-in-Lspin}, provided $A$ is simple (or else to the evaluation of a diagram of the same shape), where $M$ is the colour of the external ribbon. This is a special case of Proposition~\ref{prop:zspin_is_ribbon} below, which tells us that we can locally apply the ribbon structure of $\ALineCat$ to modify the embedded graphs.
\end{example}

Next we will define the state spaces with punctures. Note we cannot simply take the image of an idempotent as in the unpunctured case, as that would be dependent on the choice of connections $C_P$. But for any two choices there is a canonical comparison map: It is given by the evaluation of $\Sigma\times I$ with the product spin structure with two choices $C_P$ and $C_P'$ inserted at the incoming and outgoing end respectively. Note that due to the construction above an insertion of $t_M$ will have to be made if $C_P$ and $C_P'$ differ for a given puncture. We will denote the evaluation under $\Zor$ of this bordism by $p_{C_P',C_P}$. It is easy to check that
\begin{equation*}
    p_{C_P'',C_P'} \circ p_{C_P',C_P} = p_{C_P'',C_P} \ .
\end{equation*}
We will then define
\[
    \Zspin(\Sigma,\sigma,\{p_i\};P) := \lim\Big(\Zor(\Sigma;P,A^{\{p_i\}})\xrightarrow{p_{C_P',C_P}} \Zor(\Sigma;P,A^{\{p_i\}})\Big)\ .
\]
Note in particular that the taking the image of any $p_{C_P,C_P}$ for a fixed choice $C_P$ yields a representative of the limit. We will denote the map $\Zspin(\Sigma,\sigma,\{p_i\};P)\to\Zor(\Sigma;P,A^{\{p_i\}})$ coming from the universal cone by $\iota_{C_P}$ for each choice of $C_P$ and the map $\Zor(\Sigma;P,A^{\{p_i\}})\to\Zspin(\Sigma,\sigma,\{p_i\};P)$ coming from the universal property by $\pi_{C_P}$.

We can now state the main result of this paper:

\begin{theorem}\label{thm:zspin_is_a_TFT}
    Let $\Cc$ be a ribbon category and $\TargetCat$ a symmetric monoidal category, both idempotent-complete, and $\Zor:\widehat{\Bord}_3^\textsf{or}(\Cc)\to\TargetCat$ an oriented TFT respecting $\Cc$-skein relations. Fix a commutative $\Delta$-separable Frobenius algebra $A\in\Cc$. Let
    \begin{itemize}
        \item $(\Sigma,\sigma,\{p_i\};P)$ and $(\Sigma',\sigma',\{p_i'\};P')$ be objects in $\hBordsp(\ALineCat)$,
        \item $(M,E,s)$ be a morphism between the two.
    \end{itemize}

    The assignments
    \begin{align*}
        \Zspin(\Sigma,\sigma,\{p_i\};P) &:= \lim\Big(\Zor(\Sigma; P,A^{\{p_i\}})\xrightarrow{p_{C_P',C_P}}\Zor(\Sigma; P, A^{\{p_i\}})\Big),\\
        \Zspin(M,s;E) &:= \Zspin(\Sigma,\sigma,\{p_i\};P) \xrightarrow{\iota_{C_P}} \Zor(\Sigma;P,A^{\{p_i\}})\\
        & \hspace{2.75em}\xrightarrow{\Zor(M; G_s(E,C_P,C_{P'}))} \Zor(\Sigma'; P',A^{\{p_i'\}})\xrightarrow{\pi_{C_{P'}}} \Zspin(\Sigma',\sigma',\{p_i'\};P')
    \end{align*}
    define a symmetric monoidal functor
    \[
        \Zspin:\hBordsp(\ALineCat)\to\TargetCat\,.
    \]
\end{theorem}
\begin{proof}
    It is clear from previous theorems that the skeleton chosen for $M\setminus E$ is not relevant. Neither are the points of insertion for the $t_M$ by naturality: The image of the 0-cochains in the 1-cocycles is generated by cochains sending all edges incident on a single vertex to 1, which is the same as inserting a $\id_M=t_M^2$ onto a single strand incident to the corresponding coupon and pushing one of those $t_M$ through the coupon onto all other incident ribbons.

    Notice that the collection of maps $\{\Zor(M,G_s(E,C_P,C_{P'}))\}$ defines a functor from cones over $\Zor(\Sigma;P,A^{\{p_i\}})\xrightarrow{p_{C_P',C_P}} \Zor(\Sigma;P,A^{\{p_i\}})$ to cones over $\Zor(\Sigma';P',A^{\{p_i'\}}) \allowbreak \xrightarrow{p_{C_{P'}',C_{P'}}} \Zor(\Sigma';P',A^{\{p_i'\}})$. Commutativity of the relevant diagrams follows from the identity $\Zor(M,G_s(E,C_P,C_{P'}))\circ p_{C_P',C_P} = \Zor(M, G_s(E,C_P',C_{P'}))$ (and the analogous statement for the postcomposition with $p_{C_{P'}',C_{P'}}$). To verify this identity, choose a bordism representing $p_{C_P',C_P}$ which agrees with the skeleton chosen for $M$ on the boundary. Note that if the choices of $C_P$ and $C_P'$ differ this will induce a difference of one $t_M$ insertion within $p_{C_P',C_P}$. In the case of a connected component of the graph $E$ which punctures the boundary at only one point, the number of insertions on the strand going to the boundary is irrelevant, with the same argument as used above for components of $E$ that do not touch the boundary at all. In the other cases the additional $t_M$ insertion corresponds directly to the difference between $\Zor(M,G_s(E,C_P,C_{P'}))$ and $\Zor(M,G_s(E,C_P',C_{P'}))$.

    \begin{sloppypar}
    The formula given for $\Zspin(M,s;E)$ then simply gives the unique map of cones $\Zspin(\Sigma,\sigma;P)\to\Zspin(\Sigma',\sigma';P')$. The uniqueness of this map also implies that composition is well-defined.
    Note that the $A$-actions on components of the glued graph in the composition $\Zor(M',G_s(E',C_{P'},C_{P''})) \circ \Zor(M,G_s(E,C_P,C_{P'}))$ which become disconnected from the boundary can be contracted to any one of the actions.
    \end{sloppypar}
\end{proof}

Let us now check that this construction enables us to change the embedded graphs under local replacements without changing the evaluation:

\begin{proposition}\label{prop:zspin_is_ribbon}
    The gauged TFT $\Zspin$ respects $\ALineCat$-skein relations.
\end{proposition}
\begin{proof}
    Consider a 3-ball $I^3$ embedded in a bordism $M$ with an $\ALineCat$-coloured ribbon graph $E$ in such a way that it can be evaluated under $\mathcal{F}^\textsf{RT}_{\ALineCat}$. All skein relations are generated by replacing the contents of that ball with a single coupon which is coloured by the value of the ball under $\mathcal{F}^\textsf{RT}_{\ALineCat}$. We now want to reduce these relations to the $\Cc$-skein relations satisfied by $\Zor$, that is we want to show that the evaluation under $\Zor$ of $G_s(E,C_P,C_P')$ will not change by applying the relations. Since we have a TFT it is sufficient to check these relations on bordisms with underlying manifold $B^3$. In other words we have to check that sending a graph $E$ to $G_s(E,C_P,C_P')$ in $B^3$ induces a well-defined map from skein modules with $\ALineCat$-skein relations to those with $\Cc$-skein relations in $B^3$.

    To further simplify the calculation we will apply the following steps to an arbitrary $\ALineCat$-coloured graph in $B^3$ and check invariance of the evaluation:
    \begin{enumerate}
        \item Reduce the number of connected components to one by inserting identity-coloured coupons on neighbouring strands belonging to different connected components.
        \item In a given projection into two dimensions of the graph, replace all instances of crossings, twists and (co)evaluations by the respective coupons.
        \item In the resulting planar graph, compose all coupons.
    \end{enumerate}

    \begin{figure}
        \centering
        \includegraphics[]{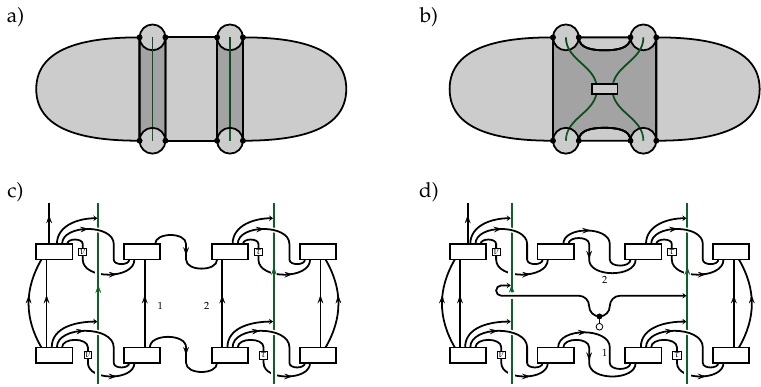}
        \caption{a) A skeleton for to two unconnected lines passing through a ball. The full 1-skeleton is given by the black lines, while the embedded $\ALineCat$ ribbons are shown in green. Some of the 2-cells (belonging to the ``equatorial plane'' of the ball) are indicated by light grey shading.
        \\
        b) A skeleton for four embedded ribbons incident on a coupon in a ball. The coupon arises from the identity in $\ALineCat$, i.e.\ it is labelled by $\iota \circ \pi$, the projector onto $- \otimes_A -$.
        \\
        c) A possible translation of the skeleton in a) into a ribbon graph.
        \\
        d) A possible translation of the skeleton in b) into a ribbon graph.
        \\
        We can transform the ribbon graph in c) into that in d) by first sliding the lines labelled 1 and 2 along the rest of the skeleton and then inserting a projector using the actions of the $A$-coloured lines on the embedded graphs.}
        \label{fig:fusion_in_b3_cells}
    \end{figure}

    For the first step we check fusing two lines in $B^3$. There are two things we need to pay attention to: First we change the topology of the graph, and thus we must change the skeleton as well. Note that the spin structure is uniquely determined by the monodromy around the lines. It remains to check that we can translate the corresponding skeleta into each other. A convenient choice of skeleton for each topology is shown in Figure~\ref{fig:fusion_in_b3_cells}.
    As shown in the figure, we slide the inner one-cells of the tubes containing the module-coloured lines in part c), which are labelled 1 and 2, along the 1-skeleton into the new positions shown in d). This works independent of any choice of translation into a ribbon graph, as correct framing is always imposed. We can use the skeleton to insert projectors between the two lines:
    \[
        \includegraphics[valign=c]{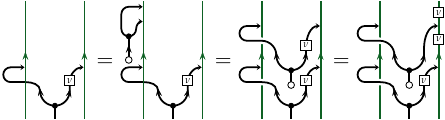}\ .
    \]
    Note the changed $t_M$ insertions in case of the framing of the original actions disagreeing. These ensure that the new paths crossing through the projector all have the correct framing with respect to our convention. Replacing the projector by a coupon yields a ribbon graph representing the skeleton shown in part b) of Figure~\ref{fig:fusion_in_b3_cells}, with an identity in $\ALineCat$ colouring the coupon: Recall that when translating coupons we compose the maps with the corresponding inclusion and projection maps and thus we get a colour of $\iota\circ\id_{M\otimes_A N}\circ\pi = p$.

    For the second step we begin by replacing twists with the appropriate coupons. Here invariance is easy to see: The difference between the twist in $\ALineCat$ and $\Cc$ is given by $t_M^{1+\deg M}$, but $1$ is precisely the difference in framing and $\deg M$ the difference due to the self-linking of $M$ after removing the twist. Second, and similarly simple is inserting coupons for (co)evaluations. Due to the insertion of (co)units when translating the graphs the $A$-actions that are part of the definition of the (co)evaluations disappear. The left (co)evaluations now agree with those in $\Cc$ (up to insertion of a projector, which works as before), while the right (co)evaluations differ by the insertion of a $t_M$ (see the list of structure morphisms below Definition~\ref{def:LA-spin}). That is precisely in line with framing change due to the insertions of the coupons of Figure~\ref{fig:2io_coupons}. Finally the last point we need to check for the second step is the braiding. Note that we alter the skeleton here, too. Invariance under change of skeleton can be argued in the same way as in the first step, by choosing two convenient skeleta and translating one into the other. Indeed, replacing a braiding by a coupon alters the self-linking in the correct way to match the $t_M$ insertion that makes up the difference between the braiding in $\ALineCat$ and $\Cc$. Lastly we can again insert projectors as needed.

    On the side of the $\ALineCat$-coloured graph we have now reduced the graph in such a way that it can be simply read off as a composition of maps. On the constructed $\Cc$-coloured graph $G_s(E,C_P,C_P')$ we would like to do the same, but we can only do so once the graph has been sufficiently unlinked from the skeleton. That can be achieved by the identity
    \[
        \includegraphics[valign=c]{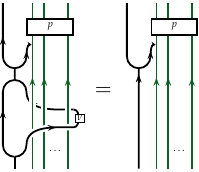}\ .
    \]
    Here, $p$ is used as a shorthand for the $A$-ribbon graph representing the idempotent projecting to $\otimes_A$. Thus we can move the skeleton out of the way and write down the composition of the coupons in $\Cc$. Due to the arguments above all coupons agree (up to maybe spurious insertions of projectors) on both sides and we arrive at the same result.
\end{proof}

\subsection{Spin Refinements}
We have seen that the insertion of $A$-coloured defects will define a spin TFT. It is clear that summing over all spin structures will generally yield a result only dependent on the underlying oriented manifolds. We will see that these operations are in fact in some sense mutual inverses. Since we will appeal to results from Section~\ref{sec:frobenius_algebras}, for the remainder of this section we put stronger assumptions on $\Cc$, $\TargetCat$ and $\Zor$:
\begin{itemize}
    \item[$k$:] a field with $\mathrm{char}(k) \neq 2$,
    \item[$\Cc$:] a $k$-linear additive idempotent-complete ribbon category, with bilinear tensor product and with absolutely simple tensor unit,
    \item[$\TargetCat:$] a $k$-linear additive idempotent-complete symmetric monoidal category,
    \item[$\Zor$:] a $k$-linear symmetric monoidal functor $\widehat{\textsf{Bord}}^\textsf{or}_3(\Cc) \to \TargetCat$, which is linear in each coupon label of a ribbon graph inside a given bordism.
\end{itemize}

Let $A$ be a haploid commutative $\Delta$-separable special Frobenius algebra in $\Cc$. By Theorem~\ref{thm:simple_commutative_frobenius_algebras}, $A$ splits into the parts $B$ and $F$. As in Remark~\ref{rem:symmetric_subalgebra_rescaling}, $B$ can also be equipped with the structure of a haploid symmetric commutative $\Delta$-separable special Frobenius algebra. Denote by $\mathcal{Z}^B$ the TFT arrived at by inserting $B$-coloured defects. It does not depend on the spin structure, as $B$ is symmetric.

Note that $\Zspin$ is defined in terms of maps $\Zor(\Sigma; A^{\{p_i\}})\to\Zor(\Sigma'; A^{\{p_i'\}})$ and that the analogous statement with $B$ in place of $A$ is true for $\mathcal{Z}^B$. For each connected surface $\Sigma$ we have a canonical maps $\Zspin(\Sigma,\sigma,p) \hookrightarrow \Zor(\Sigma;A)$ and of $\mathcal{Z}^B(\Sigma; {}_BA)\hookrightarrow\Zor(M;A,B)\cong\Zor(\Sigma,A\otimes B)$ (where by ${}_BA$ we indicate that we consider $A$ as a $B$-module here). We further have an isomorphism $\Zor(\Sigma;A)\to\Zor(\Sigma;A,B)\cong\Zor(\Sigma;A\otimes B)$ given by comultiplication of $A$ and projection onto $B$ on one of the strands, with the inverse given by multiplication.

We can now state the spin refinement theorem, which is a direct generalisation of~\cite[Thm.\,15.3]{BM96}.

\begin{theorem}\label{thm:spin_refinement}
    Let $A$ be a haploid commutative $\Delta$-separable special Frobenius algebra in $\Cc$. Assume $A$ is not symmetric and denote by $A=B\oplus F$ the splitting into the parts with twist $\pm1$. Then the functor $\Zspin$
    is a spin refinement of the functor $\mathcal{Z}^B$ in the sense that for surfaces $\Sigma,\Sigma'$, with spin structures $\sigma,\sigma'$ and marked points $\{p_i,p_i'\}$, and a bordism $M:\Sigma\to\Sigma'$ we have
    \begin{enumerate}
        \item as subspaces of $\Zor(\Sigma;(A\otimes B)^{\{p_i\}})$:
            \[
                \bigoplus_{\sigma\in\mathrm{Spin}(\Sigma;\{p_i\})} \Zspin(\Sigma,\sigma,\{p_i\}) = \mathcal{Z}^B(\Sigma; {}_BA^{\{p_i\}})\,,
            \]
        \item as morphisms $\Zor(\Sigma;B^{\{p_i\}})\to\Zor(\Sigma';B^{\{p_i'\}})$:
            \[
                2^{-b_0(M^\textsf{cl})-b_0(\Sigma)}\sum_{s\in\mathrm{Spin}(M;\{p_i,p_i'\})} \Zspin(M,s) = \mathcal{Z}^B(M)\,,
            \]
    \end{enumerate}
    where $\mathrm{Spin}(\Sigma;\{p_i\})$ and $\mathrm{Spin}(M;\{p_i,p_i'\})$ denote the sets of spin structures relative to the marked points and $b_0(M^\textsf{cl})$ is the number of connected components of $M$ without boundary.
\end{theorem}
\begin{proof}
    For the state spaces it suffices to check the statement on a connected surface.

    Note that $A$ satisfies $\mu\circ(N_A\otimes\id_A)\circ\Delta = 0$ due to Corollary~\ref{cor:even_loop_vanishes}. Thus all projectors onto the different spin state spaces are orthogonal and the sum of the $\Zspin(\Sigma,\sigma)$ can be identified with a subspace of $\Zor(\Sigma;A)$. Summing over all these projectors then gives the space $\mathcal{Z}^B(\Sigma;{}_BA)$: On all parts of the ${}_BA$-defect network except the part along the $I$-direction the sum will be over insertions of $N_A^\nu$ which is just (twice) the projector to $B$ on each edge. Due to Corollary~\ref{cor:even_loop_vanishes} only assignments of $N_A^\nu$ corresponding to a spin structure contribute, and the terms for a given $\sigma$ sum up to some non-zero multiple of the corresponding projector onto $\Zspin(\Sigma,\sigma,p)$. Along the $I$-direction there is no sum and we are just left with a single ${}_BA$ line in each connected component. Composing with the isomorphisms to and from $\Zor(\Sigma;A,B)$ then gives the identification.

    For the morphisms we begin by taking a 1-skeleton with an arbitrary framing $fr$ and converting it to a defect graph $G_{fr}\in\mathcal{R}\textsf{Frames}(fr;\{p_i\})$ as in Definition~\ref{def:framing_ribbon_class}. Let $E$ be the set of edges (without the ribbons extending to the boundary) of that graph. When we take the sum over all possible ways to insert an extra twist on the edges we obtain that each edge obtains a map $\id + N_A$, which is twice the projector onto the even part of $A$, i.e.\ $B$. The insertions on the edges will give a factor of $2^{|E|}$. The vertices, which now act as maps $B^{\otimes \ell}\to B^{\otimes k}$, each contribute a factor of $2^{-k+1}$ when adjusting the comultiplications and counits to the ones of $B$ as in Remark~\ref{rem:symmetric_subalgebra_rescaling}. On the ribbons to the boundary we can insert a projector for free, since the maps $\mu_{BB}^F$ and $\Delta_F^{BB}$ are zero. Note that we get an additional factor of $\frac{1}{2}$ for each outgoing boundary component, since these come from an additional comultiplication. Summing these up leaves us with a total factor of $2^{|V|-b_0(\Sigma')}$, where $V$ is the set of vertices.

    Only the framings corresponding to valid spin structures contribute to these sums. The argument is exactly the one of Lemma~\ref{lem:loop_removal}, except that we have the other framing on the loops we want to contract to reduce everything to the case $\mu\circ(N_A\otimes\id_A)\circ\Delta=0$.

    It remains to calculate the multiplicity with which each spin structure appears in the sum. The difference between two spin structures is described by a class in $H^1(M,\{p_i\};\Zb_2)$ and so we will have to count assignments $\nu\in\Zb_2^{|E|}$ of elements of $\Zb_2$ to one cells corresponding to the zero class in $H^1(M,\{p_i\};\Zb_2)$. These are determined by the image of the differential map $\partial_1^*:\hom(\Zb^{|V|},\Zb_2)\to\hom(\Zb^{|E|},\Zb_2)$ of the cellular cohomology of $M$, where we restrict to those cochains that vanish on the $\{p_i\}$. By definition the dimension (over $\mathbb{F}_2$) of the kernel is the dimension of $H^0(M,\{p_i\};\Zb_2)$, which is given by the number of closed connected components of $M$, $b_0(M^\textsf{cl})$. As $\dim\hom(\Zb^{|V|},\Zb_2)\big|_{c(p_i)=0} = V-b_0(\Sigma)-b_0(\Sigma')$ the image has dimension $|V|-b_0(M^\textsf{cl})-b_0(\Sigma)-b_0(\Sigma')$. Thus each spin structure appears $2^{|V|-b_0(M^\textsf{cl})-b_0(\Sigma)-b_0(\Sigma')}$ times, and cancelling the factor of $2^{|V|-b_0(\Sigma')}$ on each side we arrive at the splitting.
\end{proof}

\begin{corollary}
    If $A$ is of type $1\oplus f$ then $\Zspin$ is a spin-refinement of $\Zor$.
\end{corollary}

\begin{remark}
    In some sense the corollary is already general. By the previous theorem and with view towards Remark~\ref{rem:symmetric_subalgebra_rescaling} we may interpret the gauging of $A$ as the successive (oriented) gauging by $B$ and (spin) gauging by $A$, seen as a local $B$ module.
\end{remark}

\begin{example}
    We will give an example to illustrate the significance of the normalisation $2^{-b_0(M^\textsf{cl})-b_0(\Sigma)}$. First consider the case of $S^3$: Here the skeleton can always be contracted and we get $\mathcal{Z}^B(S^3) = \epsilon_B\circ\eta_B\cdot\Zor(S^3)=\dim B\cdot\Zor(S^3)$. If we gauge by $A$ however we find $\Zspin(S^3,s_{S^3}) = \epsilon\circ\eta\cdot\Zor(S^3)=2\dim B\cdot\Zor(S^3)$. We see that the normalisation by $2^{-b_0(M^\textsf{cl})}$ is due to the difference in (twisted) dimension of $A$ and $B$. The other term is explained by compatibility with gluing: Consider as $M_1=\bigsqcup_{i=1}^N B^3$ a collection of $N$ 3-balls with all boundaries oriented outward. There is only one relative spin structure on each ball and both $b_0(M_1^\textsf{cl})$ and $b_0(\partial_- M_1)$ are zero. We thus have $\mathcal{Z}^B(M_1) = \Zspin(M_1,s_0)$, where $s_0$ is the unique spin structure on $M_1$. Consider then $M_2=S^3\setminus \bigsqcup_{i=1}^N B^3$, a 3-sphere with $N$ 3-balls removed and boundaries oriented to be incoming. We still have $b_0(M_2^\textsf{cl})=0$, but now $b_0(\partial_- M_2)=N$. There further are $2^{N-1}$ different relative spin structures on $M_2$. We thus find $\mathcal{Z}^B(M_2)= 2^{-N}\sum_s \Zspin(M_2,s)$. After gluing $M_2\circ M_1$ the resulting manifold is just $S^3$, which should yield the normalisation $2^{-1}$. We calculate $\mathcal{Z}^B(S^3) = \mathcal{Z}^B(M_2)\circ\mathcal{Z}^B(M_1) = 2^{-N}\sum_s \Zspin(M_2\circ M_1, s\circ s_0) = 2^{-1}\Zspin(S_3,s_{S^3})$, where the last step follows from the fact that the gluing of any of the $2^{N-1}$ different spin structures on $M_2$ with the one on $M_1$ will yield the unique spin structure on $S^3$. We thus see that the factor $2^{-b_0(\Sigma)}$ adjusts for the multiplicity in the gluing of relative spin structures.
\end{example}

\section{Applications}\label{sec:applications}

We will consider two applications to Reshetikhin-Turaev type theories. First, we connect our construction to the spin TFT in~\cite{Bla03, BBC17}, which in turn generalised the Kauffman skein module construction from~\cite{BM96}. Second, we explain how to recover the classification of abelian spin Chern-Simons from~\cite{BM05} in our setting.

\subsection{Spin Modular Categories and Reshetikhin-Turaev TFTs}\label{sec:spin-modular-cat}

In this section we will see that gauging spin defects in Reshetikhin-Turaev type theory will give us invariants and state space dimensions agreeing with what was obtained in~\cite{Bla03, BBC17}. There, the input datum for the construction is a so-called spin modular category. For this section we will assume that $k$ is an algebraically closed field of characteristic 0. Recall, e.g.\ from~\cite[Ch.\,8.14]{EGNO15}, that a modular fusion category is an abelian, $k$-linear, finite semi-simple ribbon category with simple tensor unit and with non-degenerate braiding.

\begin{definition}
    A \emph{spin modular fusion category} is a modular fusion category $\Cc$ together with the choice of an object $f \in \Cc$, called the \emph{fermion}, satisfying $f\otimes f\simeq 1$ and $\theta_f=-1$.
\end{definition}

\begin{remark}
The above definition is a special case ($d=2$) of a ``modular category which is modulo $d$ spin'' \cite[Def.\,2.1]{Bla03}. A related notion where $f$ is a transparent object appeared in \cite[Thm.\,2]{Saw02}, corresponding to the subcategory $\Cc_0 \subset \Cc$ introduced below. More on spin modular categories with the additional assumption of unitarity can be found in \cite{BGHNPRW17}.
\end{remark}

\medskip

Let $\Cc$ be a spin modular category. Recall that an invertible object $x$ in ribbon category ($k$-linear and with absolutely simple tensor unit) satisfies $\dim(x) c_{x,x} = \theta_x$. For the fermion $f \in \Cc$ this implies that $c_{f,f}=1$ iff $\dim f = -1$.

The object $1\oplus f$ will always be a $\Delta$-separable Frobenius algebra, and it will be commutative if $c_{f,f}=1$, i.e.\ if $\dim f = -1$. If the dimension is $1$, we can instead take the object $f\boxtimes K^{-}\in \Cc\boxtimes\sVect$. 
Here, $\sVect$ is the ribbon category of finite-dimensional super-vector spaces over $k$ and we denote its simple objects by $K^+ = k^{1|0}$ and $K^- = k^{0|1}$. The ribbon structure on $\sVect$ is such that $K^-$ has trivial twist and quantum dimension $-1$, i.e.\ the quantum dimension is the super-dimension $\mathrm{sdim}(k^{m|n})=m-n$.
For notational convenience we will write $\widehat\Cc$ for either $\Cc$ or $\Cc\boxtimes\sVect$ and $\hat f$ for either $f$ or $f\boxtimes K^-$, depending on the dimension of $f$. We now take
\[
    A=1\oplus\hat f ~\in \widehat\Cc \ ,
\]
which by Proposition~\ref{prop:line_defects_for_1f} implies that $\ALineCat\simeq\widehat\Cc$.
This of course requires us to modify the Reshetikhin-Turaev TFT appropriately so that it can accept lines coloured in $\Cc\boxtimes\sVect$. To do so we can tensor the theory with the trivial $\sVect$-valued TFT, see~\cite[Sec.\,4]{RSW23} for details. For us the relevant upshot is that the defect lines may now be labelled in $\widehat\Cc$.

\begin{remark}
    Note that gauging algebras of the form $1\oplus\hat f$ will generate all spin theories accessible by gauging line defects in Reshetikhin-Turaev theories: By Theorem~\ref{thm:simple_commutative_frobenius_algebras} we can always see $A$ as a sum of this kind in an appropriate module category $\Cc_B^\textsf{loc}$. Due to the results of~\cite{CMRSS21} we know that gauging by $B$ will yield a Reshetikhin-Turaev type theory with defects coloured in $\Cc_B^\textsf{loc}$. The theory gauged by $A$ is then obtained by gauging $1\oplus\hat f$ in the Reshetikhin-Turaev theory for $\Cc_B^\textsf{loc}$.
\end{remark}

\medskip
We start with analysing the dimension of the state spaces. As we have seen in Theorem~\ref{thm:spin_refinement}, the state space of a genus $g$ surface $\Sigma_g$ without punctures and with spin structure $\sigma$ is a subspace of the state space of the corresponding oriented surface with a single $A$-puncture. We thus have two contributions to the overall dimension, one where the puncture is labelled $1$, which we call $\dim_+$, and one where the label is $\hat f$, which we call $\dim_-$:
\[
    \dim(\Zspin(\Sigma_g,\sigma)) \,=\, \dim_+(\Zspin(\Sigma_g,\sigma)) + \dim_-(\Zspin(\Sigma_g,\sigma))\,.
\]
Note that in the case of $\hat f = f\boxtimes K^-$, the dimension of $\dim_+$ will be the dimension of the even graded part of the state space and $\dim_-$ will be the dimension of the odd-graded part.

In~\cite{Bla03} a spin refinement of the dimension formula for the state spaces of Reshetikhin-Turaev theories coming from spin modular categories was computed. In our setting, the formula in~\cite{Bla03} computes $\dim_+$, and we extend this computation by also giving $\dim_-$. In preparation, we recall some basics about spin modular categories and refer to ~\cite{Bla03,BBC17} for details.

A spin modular category has a faithful grading induced by the braiding with $f$. $\Cc_0$ consists of those objects transparent to $f$ and $\Cc_1$ of those for which the double braiding with $f$ is given by $-1$. The Kirby colour $\omega = \sum_{\lambda\in Irr(\Cc)} \dim\lambda\cdot\lambda\in \mathcal{K}_0(\Cc)$ also splits as $\omega=\omega_0 + \omega_1$ by restricting the sum to $\Cc_0$ and $\Cc_1$ respectively.

Denote by $\mathcal{D} = \sqrt{\sum_\lambda (\dim \lambda)^2}$ the choice of square root of the categorical dimension of $\Cc$ used in the definition of the Reshetikhin-Turaev TFT.

\begin{lemma}[\cite{Bla03}, Lem.\,3.6\,\&\,3.7]\label{lem:graded_projectors}
    Let $\lambda$ be a simple element of $\Cc$. Choose a basis of $\hom(\lambda,\lambda\otimes (1 \oplus f))$ and a corresponding dual basis of $\hom(\lambda\otimes(1 \oplus f),\lambda)$ with respect to the trace pairing. Then
    \[
        \includegraphics[valign=c]{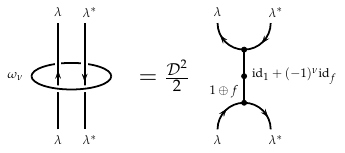} \qquad \text{and if $\lambda\otimes f\simeq\lambda$ then} \qquad \includegraphics[valign=c]{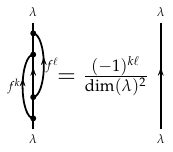}\ ,
    \]
    where unlabelled pairs of dots denote a sum over the chosen basis elements.
\end{lemma}

Note that the map inserted on the right hand side of the first equation can also be written as the twist $\theta^\nu$, and that the $f$-component is only present when $\lambda$ is an $f$ fixed point. In the second equation the sign is always positive, unless both $k$ and $\ell$ are odd, i.e.\ unless both arcs are coloured by $f$ (rather than $1$).
\begin{proposition}\label{prop:spin-RT-state-spaces}
    The dimensions of the state space of a surface $\Sigma_g$ of genus $g$ with spin structure $\sigma$ without punctures is given by, for $\varepsilon \in \{ \pm 1\}$,
    \[
        \dim_\varepsilon(\Zspin(\Sigma_g,\sigma)) = \mathcal{D}^{2g-2}\epsilon^{\frac{1-\dim f}{2}}\sum_{\lambda\in Irr(\Cc)}\varepsilon^{\deg\lambda}\dim(\lambda)^{2-2g}
            \begin{cases}
                \frac{1}{4^g}, & \lambda\not\simeq\lambda\otimes f,\\
                (-1)^{\mathrm{Arf}(\sigma)}\frac{1}{2^g}, & \lambda\simeq\lambda\otimes f.
            \end{cases}
    \]
\end{proposition}
\begin{proof}
    \begin{figure}
        \centering
        \includegraphics[width=\textwidth]{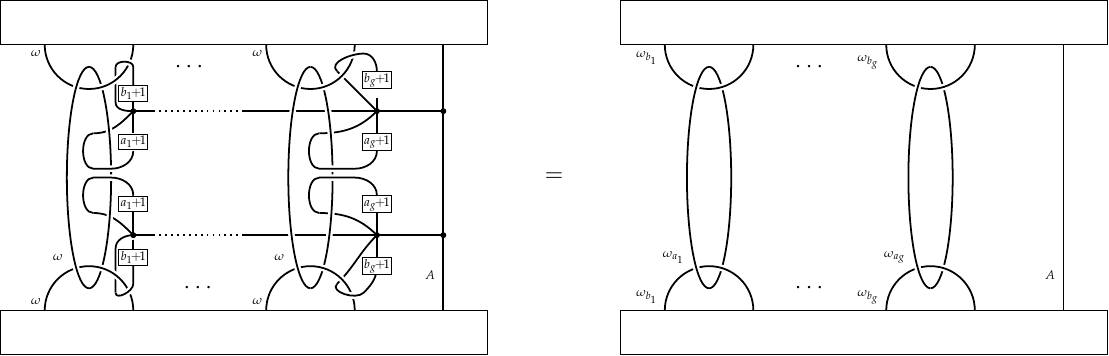}
        \caption{On the left hand side is (up to constant factors of powers of $\Dc$) the projector for the state space of $\Sigma_g$. Shown is a slight rewriting of the ribbon graph corresponding to the 1-skeleton with only one 0-cell, which arises by pulling apart the (co)multiplications which should be inserted at the 0-cell. The dots on $A$-coloured lines are the appropriate compositions of the remaining (co)multiplications, while the boxes indicate insertions of the Nakayama with powers $a_i+1$ or $b_i+1$. On the right hand side we see what remains after contracting the $A$-skeleton as far as possible. The resulting graph (without the $A$-ribbon) is also called the special surgery graph for $\Sigma_g \times I$. The colourings of the special surgery graph are projected onto the graded Kirby colours according to the values of $a_i$ and $b_i$.}
        \label{fig:state_space_projector}
    \end{figure}
    \begin{figure}
        \centering
        \includegraphics{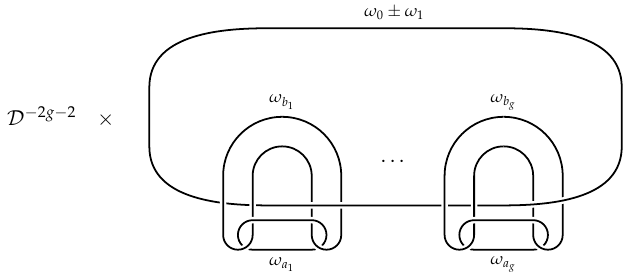}
        \caption{The ribbon graph representing the trace over the even ($\omega_0+\omega_1$) and odd ($\omega_0-\omega_1$) parts of the projector.}
        \label{fig:s1_times_surface}
    \end{figure}
    \begin{figure}
        \centering
        \includegraphics{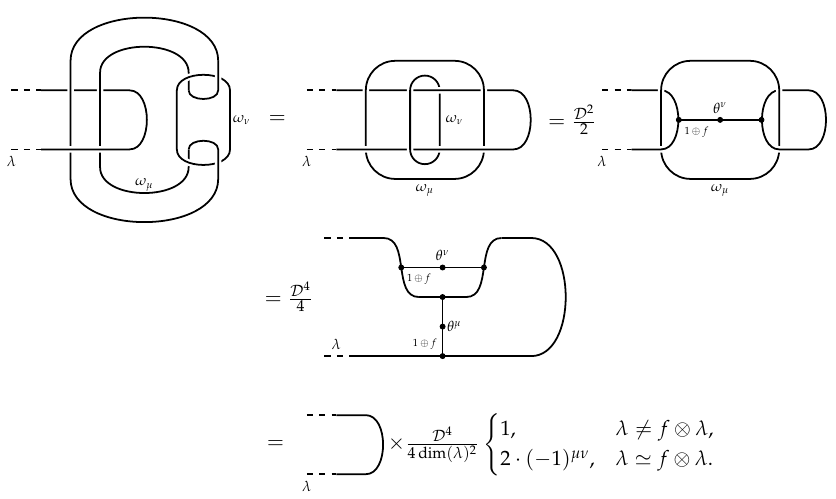}
        \caption{A simple object $\lambda$ linking part of the special ribbon graph of $\Sigma_g\times S^1$. The calculation uses the first part of Lemma~\ref{lem:graded_projectors} for the second and third equality and the second part for the fourth.}
        \label{fig:state_space_calc}
    \end{figure}
    Take the standard ribbon graph giving the surgery presentation of $\Sigma_g\times I$ and take with it the one-skeleton of $\Sigma\times I$ with only one 0-cell, as shown on the left hand side of Figure~\ref{fig:state_space_projector}. Given a spin structure on $\Sigma_g$ we can express this spin structure in terms of the symplectic basis, see e.g.~\cite{Bel98}. A spin structure on a genus $g$ surface is then a $2g$-tuple $(a_i,b_i)_{i=1,\dots,g}\in {(\Zb_2^2)}^g$, where each coefficient describes how the framing on one of the standard generators of $\pi_1(\Sigma_g)$ behaves in reference to a standard framing. The 1-skeleton is chosen in such a way that the loops coincide with the standard set of generators of $\pi_1(\Sigma_g)$, and therefore such that the coefficients $a_i,b_i\in\Zb_2$ describe the homotopy class of framing of each loop.

    We can unlink the $A$-coloured ribbon graph from the surgery link by using the relation
    \[
        c_{A,\omega}\circ c_{\omega,A} = \id_{\omega_0}\otimes\id_A \oplus \id_{\omega_1}\otimes N_A\,.
    \]
    The resulting closed loops in the $A$-ribbon graph will only be non-zero for the Kirby colours $\omega_{a_i}$ and $\omega_{b_i}$ as determined by the spin structure, see the right hand side of Figure~\ref{fig:state_space_projector}.

    To now find the dimension of the state space we project to the tensor unit by linking the special ribbon graph of $\Sigma_g\times I$ with $\omega=\omega_0+\omega_1$ and to $f$ by linking with $\omega_0-\omega_1$. Taking the trace will now give the categorical dimensions. In the latter case we also have to account for the dimension of $\hat f$ when taking the trace which yields an additional factor of $-1$. When we are $\sVect$-valued this contributes an odd super-vector space due to $\hat f = f\boxtimes K^-$. Since we are interested in the number of basis elements (rather than the categorical dimension) we will want to adjust $\dim_-$ by another factor of $-1$ in the case of $\dim f =1$. The total sign we need to multiply the evaluation of Figure~\ref{fig:s1_times_surface} by can then be expressed succinctly as $\epsilon^{\frac{1-\dim f}{2}}$.

    From this point the calculation can essentially already be found in~\cite[Thm.\,3.3]{Bla03} for $\dim_+$. Namely, the ribbon graph computing the trace is given in Figure~\ref{fig:s1_times_surface}. Applying the identity in Figure~\ref{fig:state_space_calc} for all $g$ parts of the diagram. In each step, the signs $\mu,\nu$ for the graded Kirby colours are $\nu=a_i$ and $\mu=b_i$, so that the overall sign is
    \[
        (-1)^{\sum_{i=1}^g a_ib_i} = (-1)^{\mathrm{Arf}(\sigma)}\,.
    \]
    \phantom{.}
\end{proof}

\begin{remark}
\begin{enumerate}[leftmargin=*]
    \item Any fixed point $\lambda \in Irr(\Cc)$ of $f$ has to be in the odd component: $\lambda \in \Cc_1$. This follows from applying the twist to $\lambda \otimes f \cong \lambda$ together with $\theta_{f \otimes \lambda} = c_{\lambda,f} \circ c_{f,\lambda} \circ (\theta_f \otimes \theta_\lambda)$. We can use this to write the state space dimensions as
    \[
        \dim_\varepsilon(\Zspin(\Sigma_g,\sigma)) = \epsilon^{\frac{1-\dim f}{2}}(D_0 + \epsilon D_1)\ ,
    \]
    where
    \begin{align*}
        D_0 &= \mathcal{D}^{2g-2} 4^{-g}  \sum_{\lambda\in Irr(\Cc_0)} \dim(\lambda)^{2-2g} \ ,\\
        D_1 &= \mathcal{D}^{2g-2} \sum_{\lambda\in Irr(\Cc_1)} \dim(\lambda)^{2-2g}
            \begin{cases}
                \frac{1}{4^g}, & \lambda\not\simeq\lambda\otimes f,\\
                (-1)^{\mathrm{Arf}(\sigma)}\frac{1}{2^g}, & \lambda\simeq\lambda\otimes f.
            \end{cases}
    \end{align*}
    In particular, the total dimension
    \[
        \dim_+(\Zspin(\Sigma_g,\sigma)) + \dim_-(\Zspin(\Sigma_g,\sigma)) = \begin{cases} 2D_0, & \dim f = 1,\\ 2D_1, &\dim f = -1,\end{cases}
    \]
    does not depend on the spin structure $\sigma$, since for $\dim f=-1$ there can be no fixed points under tensoring with $f$.

    \item When $f$ has no fixed points, the dimensions $\dim_+ \Zspin(\Sigma,\sigma)$ and $\dim_- \Zspin(\Sigma,\sigma)$ are independent of the spin structure $\sigma$. This is always true when we have $\dim f = -1$. Having a fixed point under $f$ shifts a contribution from $\dim_+$ to $\dim_-$, that is to an odd vector space, for surfaces with Arf-invariant 1. Whether the $\dim_-$ must vanish in the absence of $f$ fixed points is unclear to us and is equivalent to the question of whether spin modular categories without $f$ fixed points necessarily have a grading with three or more components (in particular such that $f$ is not in the trivial component) posed in~\cite[Question 2.19]{BGHNPRW17}.

\end{enumerate}
\end{remark}

\begin{example}\label{ex:pointed_spin_state_space}
\begin{enumerate}[leftmargin=*]
    \item
If the spin modular category $\Cc$ is pointed with simple objects indexed by some finite abelian group $G$, the formula in Proposition~\ref{prop:spin-RT-state-spaces} simplifies to
    \[
        \dim_+(\Zspin(\Sigma_g,\sigma)) = \big( |G|/4 \big)^{g}~,\quad \dim_-(\Zspin(\Sigma_g,\sigma)) = 0\,.
    \]
In particular, the dimensions $\dim_\pm$ of the state space do not depend on the spin structure $\sigma$ on $\Sigma_g$.

\item
The other extreme is the Ising category. Here $\Cc_0 = \{ 1,\epsilon\}$ and $\Cc_1 = \{ \sigma \}$ with $f = \epsilon$ and $\sigma$ a fixed point. Thus $\Cc_1$ consists entirely of fixed points for $f$. The formula in Proposition~\ref{prop:spin-RT-state-spaces} gives, for $\nu = \pm 1$,
    \begin{align*}
        \dim_{\nu}(\Zspin(\Sigma_g,\sigma))
        = 2^{2g-2} \Big( 2 \, \frac{1}{4^g} + \nu (\sqrt{2})^{2-2g} (-1)^{\mathrm{Arf}(\sigma)}\frac{1}{2^g}\Big)
        = \begin{cases}
            1\,, & (-1)^{\mathrm{Arf}(\sigma)} = \nu \ ,\\
            0\,, & (-1)^{\mathrm{Arf}(\sigma)} \neq \nu \ .
        \end{cases}
    \end{align*}
We see that all state spaces $\Zspin(\Sigma_g,\sigma)$ are one-dimensional, and that their parity is given by $\mathrm{Arf}(\sigma)$. This agrees with the expected results from~\cite[Sec.\,7.2]{GK16}.
\end{enumerate}
\end{example}

We will next compare the gauging to the three-manifold invariants from~\cite{Bla03,BBC17}.

Let $M_L$ be a manifold obtained by surgery on the link $L$. Spin structures on $M_L$ can be expressed in terms of a characteristic sublink, membership in which determines whether a frame running in a loop around the link component will induce a deck transformation in the spin cover. For details see~\cite{Bla92,BM96}. Let then $s=\{s_1,\dots,s_{|L|}\}$ be a spin structure on $M_L$, where the $s_i$ are 1 if the corresponding component of the link is part of the characteristic sublink of $L$ and 0 else.
\begin{theorem}[\cite{Bla03}, Thm.\,2.2]
    The number
    \[
        \tau_\Cc(M_L,s) = \frac{\Fc_\Cc(L(\omega_{s_1},\dots,\omega_{s_{|L|}}))}{\Fc_\Cc(U_+(\omega))^{b_+}\Fc_\Cc(U_-(\omega))^{b_-}}
    \]
    is an invariant of $M_L$ and $s$. Here $\Fc_\Cc$ is the Reshetikhin-Turaev invariant associated to $\Cc$-coloured links in $S^3$, $b_\pm$ is the number of positive (negative) eigenvalues of the linking matrix $L_{ij}$ and $U_\pm$ are unknots with framing $\pm 1$.
\end{theorem}

\begin{proposition}\label{prop:spin-TFT-produces-MF-invariants}
    Let $L$ be a surgery link and $M_L$ the closed $3$-manifold derived by surgery on the link $L$ in $S^3$. Let $s$ be a spin structure on $M$, which we may identify with the sequence $s_i$ determining the characteristic sublink of $L$.
    We have
    \[
        \Zspin(M_L,s) = 2 \mathcal{D}^{-1-b_1} \tau_\Cc(M_L,s)\,.
    \]
\end{proposition}
\begin{proof}
    We will first deal with the normalisation. The invariants $\tau_\Cc(M_L,s)$ normalise the evaluation of the link by $\mathcal{D}_+^{-b_+}\mathcal{D}_-^{-b_-}$ (where $\mathcal{D}_\pm$ is the evaluation of the unknot with framing $\pm1$). In the case of the Reshetikhin-Turaev invariants, however, the normalisation is $\mathcal{D}^{-1-|L|}\delta^{-\sigma(L)}$, where $\sigma(L)$ is the signature of the linking matrix of $L$ and $|L|$ is the number of components in the link, $\mathcal{D}$ is a square root of $\mathcal{D}_+\mathcal{D}_-$ and $\delta=\frac{\mathcal{D}_+}{\mathcal{D}}=\frac{\mathcal{D}}{\mathcal{D}_-}$.
    We calculate
    \begin{align*}
        \mathcal{D}_+^{-b_+}\mathcal{D}_-^{-b_-}
        &= \mathcal{D}^{-b_+-b_-} \delta^{-\sigma(L)}.
    \end{align*}
    It remains to relate $\mathcal{D}^{-1-|L|}$ and $\mathcal{D}^{-1-b_1-b_+-b_-}$. To do so we note, that the $1$-homology classes of $M_L$ correspond precisely to the $0$-eigenvalues of the linking matrix $L_{ij}$, and so $b_1+b_++b_-=|L|$. To see this we can use that
    \[
        H_1(M_L,\Zb) = \Zb^{|L|} \Big/ \sum_j L_{ij}\mu_j\,,
    \]
    where the $\mu_i$ are the generators of the $\Zb$ factors. See~\cite[Ch.\,5]{GS99}.

    Before we continue we should compare the two descriptions of spin structures -- namely via characteristic sublinks and via framed 1-skeleta -- and show how to translate one to the other. To do so we consider the surgery link presentation of $M_L$ in $S^3$. The translation of the spin structure into the 1-skeleton gives us a ribbon graph in $M_L$ which in turn gives us a ribbon graph $G_s$ in $S^3\setminus L$. Given a spin-structure characterised via a sublink we can deduce the (homotopy class of) framing on the 1-skeleton: For a loop $\ell_i$ in the 1-skeleton we have that the framing (as a ribbon graph in $S^3$) $p(\ell_i)$ of $\ell_i$ must satisfy:
    \[
        p(\ell_i) \equiv \sum_j \mathrm{Link}(\ell_i,L_j)s_j + 1 \mod 2\,.
    \]
    In particular this means that a loop not linking any parts of $L$ will have the contractible framing in $S^3$, too. See~\cite{Bel98} for details.

    For the inverse we take the above relation as a system of linear equations which we have to solve for $s_i$.

    With that we return to showing that $2\Fc_\Cc(L(\omega_{s_1},\dots,\omega_{s_{|L|}})) = \Fc_\Cc(L(\omega,\dots,\omega) \cup G_s)$. Initially all parts of $L$ are coloured by $\omega$. We split this into a sum over the colourings $\omega_0$ and $\omega_1$:
    \[
        \Fc_\Cc(L(\omega,\dots,\omega)\cup G_s) = \sum_{(s_1,\dots,s_{|L|})\in\Zb_2^{|L|}} \Fc_\Cc(L(\omega_{s_1},\dots,\omega_{s_{|L|}})\cup G_s)\,.
    \]
    We can now first unlink the 1-skeleton from $L$ using the braiding given above and then contract the entire skeleton to a single coupon using Lemma~\ref{lem:loop_removal}. The coupon will evaluate to 2 if all loops were contractible and to 0 else due to Corollary~\ref{cor:even_loop_vanishes}.

    When we braid the $A$-coloured strand through an $\omega_i$-coloured strand the framing of the $A$-coloured strand changes by $i$. Denote by $p(\ell_i)$ the framing of a loop $\ell_i$ in the 1-skeleton relative to $S^3$. Then the unlinked loop has framing
    \[
        p(\ell_i) + \sum_j \mathrm{Link}(\ell_i,L_j) s_{j} \mod 2\,.
    \]
    It needs to be odd for the loop to be contractible. That is precisely the condition for the $s_j$ to encode the same spin structure as the framing and we see that only this term in the sum can survive.
\end{proof}

\begin{example}
    As an example we will take the three-torus. A PLCW decomposition of the three-torus is easily derived from the picture of gluing up the opposing sides of a cube and a convenient surgery link is given by the Borromean link with $0$-framings. The resulting graph and its simplification are shown in Figure~\ref{fig:t3_rt_example}. The blue link will now yield the manifold invariant from~\cite{BBC17}, while the black part will give the twisted dimension of $A$, which is 2.

\begin{figure}
    \centering
    \includegraphics[width=0.8\linewidth]{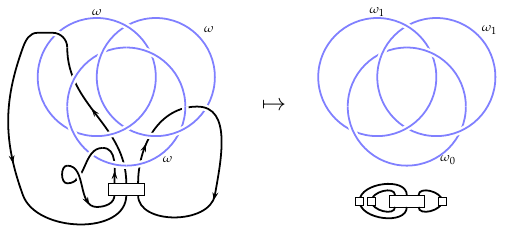}
    \caption{Simplification of the torus link with added spin defect network.}
    \label{fig:t3_rt_example}
\end{figure}

    We can calculate the result in general. By a calculation essentially the same as for the state spaces we find
    \[
        \Zspin(T^3,s) = 2\sum_{\lambda\in Irr(\Cc)}\begin{cases}
            \frac{1}{4}, & \lambda\not\simeq f\otimes\lambda,\\
            \frac12 (-1)^{s_1s_2s_3} , & \lambda\simeq f\otimes\lambda.
        \end{cases}
    \]
    In particular we find that all spin structures for which at least one $s_i$ is zero will always yield the same result. This is to be expected, as these are diffeomorphic. The one spin structure not diffeomorphic to the others will be distinguished if we have at least one fixed point under $f$.
\end{example}

\begin{example}\label{ex:val-in-vect}
    To show that the TFTs that take values in $\Vect$ can also produce true spin TFTs we consider the Borromean link where we change the framing of each component to two. There are -- again -- eight spin structures on the manifold produced by surgery on that link. Take $\Cc$ to be the braided fusion category with fusion rules given by $\Zb_2\times\Zb_2$ and braided structure induced by the quadratic form (see the next section for more on these)
    \begin{align*}
        q: \Zb_2\times\Zb_2 \to \Cb^\times,\quad
        (0,0) \mapsto 1,\quad
        (1,1) \mapsto 1,\quad
        (1,0) \mapsto i,\quad
        (0,1) \mapsto -i.
    \end{align*}
    We also assign dimension $1$ to $(0,0)$ and $(1,0)$ while $(0,1)$ and $(1,1)$ will have dimension $-1$. The corresponding $T$- and $S$-matrices are given by (in the basis $(0,0),(1,1),(1,0),(0,1)$)
    \[
        S = \begin{pmatrix} 1 & -1 & 1 & -1\\ -1 & 1 & 1 & -1\\ 1 & 1 & -1 & -1\\ -1 & -1 & -1 & -1\end{pmatrix}\ , \qquad T = \begin{pmatrix} 1 & & & \\ & -1 & & \\ & & i &\\ & & & i\end{pmatrix}\ ,
    \]
    telling us $\Cc$ is modular. Then $\Cc$ is spin modular with fermion $f=(1,1)$. To evaluate $\Zspin$ on $M$ note that $\theta^2_{\omega_0} = \id_{\omega_0}$ and $\theta^2_{\omega_1}=-\id_{\omega_1}$. We can thus remove the twists from the surgery graph at the cost of an overall constant factor of $\pm1$ and thereby reduce the calculation to the evaluation on the torus already given above. Thus
    \[
        \Zspin(M,s) = (-1)^{s_1+s_2+s_3}\,.
    \]
\end{example}

\subsection{Abelian Spin Chern-Simons Theories}\label{subsec:abelian_cs}

In~\cite{BM05} the classification of abelian spin Chern-Simons TFTs is given in terms of certain metric groups. These are finite abelian groups $G$ together with a \emph{quadratic form} $q:G\to\Qb/\Zb$, i.e. a map $q$ such that $b(x,y):=q(x+y)-q(x)-q(y)+q(0)$ is bilinear. A quadratic form $q$ is \emph{non-degenerate} if $b$ is non-degenerate.
    We call a quadratic form $q$ \emph{homogeneous} if $q(nx) = n^2 q(x)$ for $n \in \mathbb{Z}$.\footnote{A homogeneous quadratic form automatically satisfies $q(0)=0$, but the normalisation condition $q(0)=0$ does not imply that $q$ is homogeneous. Instead, $q(x) - q(-x)$ can be a non-trivial linear function, see e.g.\ \cite[Sec.\,5.3]{HS05}.}

To a quadratic form $q$ we can assign an equivalence class $[q]$ given by the relation that $q\sim q'$ when there exists an element $\Delta\in G$ such that $q(g) = q'(g-\Delta)$ for all $g\in G$.

Associated to a group $G$ with quadratic form $q$ we have the Gauß sum
\[
    \tau^\pm(G,q) := \frac{1}{\sqrt{|G|}} \sum_{g\in G} e^{\pm 2\pi i q(g)}\,,
\]
which is a root of unity. Note that the Gauß sum is equal for all quadratic forms in the same equivalence class.

It is shown in~\cite{BM05} that abelian spin Chern-Simons theories are classified by tuples $(\mathscr{D},[q],\sigma)$, where
\begin{itemize}
    \item $\mathscr{D}$ is a finite abelian group,
    \item $[q]$ is an equivalence class of non-degenerate quadratic forms $q:\mathscr{D}\to\Qb/\Zb$,
    \item $\sigma\in\Zb_{24}$.
\end{itemize}
This data is additionally subject to the Gauß-Milgram constraint, which stipulates that the Gauß sum be given by
\[
    \tau^+(\mathscr{D},q) = e^{2\pi i \sigma / 8}\,,
\]
which also fixes the normalisation $q(0)$. If there exists a representative of $[q]$ with $q(0)=0$ the data corresponds to an oriented theory. We will call the tuple spin if that is not the case.

\begin{definition}
    We define the groupoid of \emph{abelian spin Chern-Simons data} $\textsf{ASCS}$ as having
    \begin{itemize}
        \item objects given by spin tuples $(\mathscr{D},[q],\sigma \mod 8)$ with $\mathscr{D}$ as above, and
        \item morphisms $(\mathscr{D}_1,[q_1],\sigma_1)\to(\mathscr{D}_2,[q_2],\sigma_2)$ given by group isomorphism $\phi:\mathscr{D}_1\to\mathscr{D}_2$ such that $[q_2\circ\phi]=[q_1]$.
    \end{itemize}
\end{definition}
Note that we only take $\sigma\in\Zb_8$ within the groupoid\footnote{The value modulo 24 determines a deprojectification of the mapping class group action associated to the TFT, which is not accessible to our TFT-based comparison.} and that the existence of a morphism already implies $\sigma_1=\sigma_2$.

    Next we turn to pointed spin modular categories. A pointed modular category is described -- up to equivalence -- by an abelian group $G$ together with a non-degenerate homogeneous quadratic form $\hat q: G\to \Qb/\Zb$. See~\cite[Sec.\,8.4]{EGNO15} for this construction,
    but note that there ``quadratic form'' denotes what we call homogeneous quadratic form, and that these are written multiplicatively as $q:G\to\Cb^\times$.

    In this subsection we assume that the quantum dimensions of the simple objects are all $+1$, else one needs to include a group homomorphism $G \to \Zb_2$ as extra data. In particular, Example~\ref{ex:val-in-vect} is excluded in this setting, and, as we saw in Section~\ref{sec:spin-modular-cat}, the algebra to be gauged is necessarily $A = 1 \oplus \hat f \in \Cc\boxtimes\sVect$ with $\hat f = f\boxtimes K^{-}$.

\begin{definition}
    We define the groupoid of \emph{pointed spin-modular categories} $\textsf{PSM}$ as having
    \begin{itemize}
        \item objects given by tuples $(G,\hat q,f)$ with $G$ a finite abelian group, $\hat q$ a non-degenerate homogeneous quadratic form, and $f$ a fermion, i.e.\ $2f=0$ and $\hat q(f)=\frac{1}{2}$ (and so in particular $f \neq 0$),
        \item morphisms $(G_1,\hat q_1,f_1)\to (G_2,\hat q_2,f_2)$ given by group isomorphisms $\phi: G_1\to G_2$ such that $\hat q_2\circ\phi = \hat q_1$ and $\phi(f_1)=f_2$.
    \end{itemize}
\end{definition}

To declutter the notation we will commonly only refer to the groups when referencing objects from either groupoid, leaving the additionally data implicit.

As we will see, the groupoids $\textsf{ASCS}$ and $\textsf{PSM}$ are almost equivalent.

We define a functor $F:\textsf{PSM}\to\textsf{ASCS}$ as follows: Given $(G,\hat q, f)\in\textsf{PSM}$, let $\hat b$ be the bilinear form associated to $\hat q$. We set $G_0:=\{g\in G\, |\, \hat b(g,f)=0\}$ to be the subgroup transparent with respect to $f$ and define $\mathscr{D} := G_0/\langle f\rangle$. Choose a representative $a\in[1]\in G/G_0 \cong \Zb_2$ and set $[q]:=[q_a]$, where $q_a(x):= \hat q(x-a)$. Lastly we determine $\sigma$ via the Gauß sum by demanding $\tau^+(G,\hat q) = e^{2\pi i\sigma/8}$.

Applied to an isomorphism $\phi: G_1\to G_2$, the functor $F$ will simply give the induced isomorphism on the quotients.

\begin{proposition}\label{prop:spin-abelian-Chern-Simons}
$F$ defines a functor $\textsf{PSM}\to\textsf{ASCS}$ which is essentially surjective and 2:1 on morphisms. It thus induces a bijection $\pi_0(\textsf{PSM})\cong\pi_0(\textsf{ASCS})$.
\end{proposition}

We will prove this proposition by series of lemmas.

\begin{lemma}
    The assignments above make $F$ a well-defined functor.
\end{lemma}
\begin{proof}
    We first check that $q_a$ descends to a well-defined quadratic form on $\mathscr{D}$, and that different choices of $a$ lead to equivalent quadratic forms on $\mathscr{D}$. Calculate, for $x \in G_0$,
    \begin{align*}
        \hat q(x+f-a) &= \hat b(x+f,-a) - \hat q(x+f) - \hat q(-a)\\
            &= \hat b(x,-a) + \underbrace{\hat b(f,-a)}_{=\frac{1}{2}} - \underbrace{\hat b(x,f)}_{=0} - \hat q(x) - \underbrace{\hat q(f)}_{=\frac{1}{2}} - \hat q(-a)\\
            &= \hat b(x,-a) - \hat q(x) - \hat q(-a)\\
            &= \hat q(x-a).
    \end{align*}
    This shows that $q_a$ is well-defined on $\mathscr{D}$ and it further shows that any shift of representative $a$ by an element in $G_0$ corresponds to a shift of $q_a$ by an element in $\mathscr{D}$ which will therefore be in the same class.

    Next we check that $q_a$ satisfies the Gauß-Milgram constraint with respect to the $\sigma$ we determined above. That the Gauß sum must be an eighth root of unity is shown by arguments in~\cite[Sec.\,6]{BM05}: The group $G$ can be realised as the discriminant group of some (even) lattice such that $\hat q$ is its discriminant form. Then $\sigma$ is the signature of the lattice modulo 8. It remains to check that the Gauß sums on $\mathscr{D}$ and $G$ agree.

    Choose a representative $a$ of the generator of $G/G_0$ as above:
    \begin{align*}
        \frac{1}{\sqrt{|G|}}\sum_{g\in G} e^{2\pi i \hat q(g)} &= \frac{1}{\sqrt{|G|}}\Big(\underbrace{\sum_{g\in G_0} e^{2\pi i \hat q(g)}}_{=0} + \sum_{g\in G\setminus G_0} e^{2\pi i \hat q(g)}\Big)\\
            &=\frac{1}{\sqrt{|G|}}\sum_{g\in G_0} e^{2\pi i \hat q(g-a)}
            = \frac{2}{\sqrt{|G|}}\sum_{[g]\in\mathscr{D}} e^{2\pi i \hat q(g-a)}\\
            &= \frac{1}{\sqrt{|\mathscr{D}|}}\sum_{x\in\mathscr{D}} e^{2\pi i q_a(x)}.
    \end{align*}
    Here the first sum on the right hand side vanishes due to $\hat q(g) = \hat q(f+g)+\frac{1}{2}$ for $g\in G_0$.
\end{proof}

\begin{lemma}\label{lem:PSM-ASCS-ess-surj}
    The functor $F:\textsf{PSM}\to\textsf{ASCS}$ is essentially surjective.
\end{lemma}
\begin{proof}
    It was shown in~\cite[Sec.\,6]{BM05} that any tuple $(\mathscr{D},[q],\sigma)$ arises from the data $(\Lambda, [W_2])$, where $\Lambda$ is a lattice with integral bilinear form $(\cdot,\cdot)$ of signature $\sigma$, $W_2\in\Lambda^*/2\Lambda^*$ is a characteristic class satisfying $(W_2,\lambda)\equiv (\lambda,\lambda)\mod 2$ for $\lambda\in\Lambda$, and $\mathscr{D}=\Lambda^*/\Lambda$ is the discriminant group. Then a given representative $W_2$ will define a representative of $[q]$ via $q([x]) = \frac{1}{2}(x,x-W_2)+\frac{1}{8}(W_2,W_2)\mod 1$, which is well-defined due to the defining property of $W_2$. We will consider $\Lambda$ as living in some ambient space $\Qb^N$.

Given that by assumption $q(0)\neq 0$ for all representatives $q$ of $[q]$, we have that $[W_2]\neq[0]$. Define $\Lambda_0 := \ker \left(\Lambda\to\Zb_2,\ \lambda\mapsto (W_2,\lambda)\mod 2\right)$ as the even sublattice. One easily verifies $\Lambda_0^* = \mathrm{span}_\Zb (\Lambda^*, \frac{W_2}{2})$. Now set $G:=\Lambda_0^*/\Lambda_0$ and $\hat q([x]) := \frac{1}{2}(x,x)$. We identify $f$ as the non-trivial element of $\Lambda/\Lambda_0\simeq\Zb_2$ which gives a well-defined element of $G$ that satisfies $2f=0$ and $\hat q(f)=\frac{1}{2}$.

    Next we check that that $\mathscr{D} \cong G_0/\langle f\rangle$. Note that for any lift $\tilde f$ of $f$ to $\Lambda_0^*$ we have $(\tilde f,x)\in\Zb$ for all elements $x\in\Lambda^*$ and $(\tilde f,\frac{W_2}{2})\in \frac{1}{2} + \Zb$. It follows that $G_0$ may be identified as $\Lambda^*/\Lambda_0$ and thus $\mathscr{D} = \Lambda^*/\Lambda \cong (\Lambda^*/\Lambda_0) / (\Lambda / \Lambda_0) = G_0/\langle f\rangle$.

    As shown above, $\frac{W_2}{2}$ is not in $G_0$ and may hence be taken as a representative of the non-trivial class in $G/G_0$. By definition we then have (all taken modulo $1$)
    \[
        q_{\frac{W_2}{2}}([x]) = \hat q(x - \tfrac{W_2}{2}) = \tfrac{1}{2}(x-\tfrac{W_2}{2},x-\tfrac{W_2}{2}) = \tfrac{1}{2}(x,x-W_2) + \tfrac{1}{8}(W_2,W_2) \ ,
    \]
    recovering the original $q$. Lastly, simply imposing the Gauß-Milgram constraint will recover $\sigma$, as $q$ satisfied it to begin with.
\end{proof}
\begin{lemma}
    The functor $F$ is full.
\end{lemma}
\begin{proof}
Let $(G_i,\hat q_i,f_i) \in \textsf{ASCS}$, $i=1,2$, be given, and write $(\mathscr{D}_i,[q_i],\sigma_i) = F(G_i,\hat q_i,f_i)$. Take a morphism $\phi:(\mathscr{D}_1,[q_1],\sigma_1) \to (\mathscr{D}_2,[q_2],\sigma_2)$. We will construct a preimage of $\phi$ (the dashed arrow below) by realising
    \[
        \begin{tikzcd}[column sep=tiny, row sep=tiny]
            G_{\Lambda'_0} \ar[ddrr] & G_1 \ar[l]\ar[rr,dashed] & & G_2 \ar[r] & G_{\Lambda''_0} \ar[lldd] & & & & & \mathscr{D}_{\Lambda'}\ar[ddrr] & \mathscr{D}_1\ar[l]\ar[rr,"\phi"] & & \mathscr{D}_2\ar[r] & \mathscr{D}_{\Lambda''}\ar[lldd] \\
            & & & & &  \phantom{.}\ar[rrr, "F", mapsto] & & &\phantom{.} & & & & & & ,\\
            & & G_{\Lambda_0} & & & & & & & & & \mathscr{D}_\Lambda & &
        \end{tikzcd}
    \]
    where the left hand side is a lift of the commutative diagram on the right hand side. Since all morphisms are isomorphisms, constructing all non-dashed arrows on the right hand side will suffice. The subscripts on the groups indicate that the corresponding elements of $\textsf{PSM}$ or $\textsf{ASCS}$ are constructed as discriminant groups of the lattice in the subscript. The top left (and top right, by analogy) parts of the triangles are constructed by choosing even lattices $\Lambda'_0$ and $\Lambda''_0$ inducing $G_{\Lambda'_0}\cong G_1$ and $G_{\Lambda''_0}\cong G_2$. We get $\Lambda'$ and $\Lambda''$ by adding a representatives of $f$ and $[a] \in G_i/(G_i)_0$ to the respective lattice.

    Next we construct $\mathscr{D}_\Lambda$. By construction, the discriminant groups $\mathscr{D}_{\Lambda'}$ and $\mathscr{D}_{\Lambda''}$ are isomorphic in a way compatible with their quadratic forms and signature modulo $8$. Hence, the lattices $\Lambda'$ and $\Lambda''$ are stably equivalent (that is there exist unimodular lattices $U,V$ such that $\Lambda'\oplus U\cong\Lambda''\oplus V$), see~\cite[Thm.\,1.3.1]{Nik79}. Since $\Lambda'$ and $\Lambda''$ had the same signature (modulo 8) to begin with, we know that $U$ and $V$ have the same signature (modulo 8), which we may take to be $0$ (extending trivially by adding copies of $\Zb$ on both sides, if necessary). Set then $\Lambda:=\Lambda'\oplus U$.

    There is the obvious isomorphism $\mathscr{D}_{\Lambda'}\to\mathscr{D}_\Lambda$ induced by the inclusion $\Lambda'\hookrightarrow\Lambda$ giving the left diagonal arrow of the right triangle. To construct the right diagonal arrow we also begin by taking the isomorphism $\mathscr{D}_{\Lambda''}\to\mathscr{D}_{\Lambda''\oplus V}$ induced by the inclusion $\Lambda''\hookrightarrow \Lambda''\oplus V$. We then have the isomorphism $h:\mathscr{D}_{\Lambda''\oplus V} \to \mathscr{D}_{\Lambda''} \to \mathscr{D}_{2} \xrightarrow{\phi^{-1}} \mathscr{D}_{1} \to \mathscr{D}_{\Lambda'} \to \mathscr{D}_\Lambda = \mathscr{D}_{\Lambda'\oplus U}$. That isomorphism can be lifted to the lattices by the results of~\cite[Thm.\,1.16.10]{Nik79} (increasing the rank of $U$ and $V$ by a unimodular signature zero lattice if necessary to meet the conditions in the theorem). The composition of $\Lambda''\hookrightarrow \Lambda''\oplus V$ with the lifted isomorphism then gives the right diagonal arrow in the right diagram, making the diagram commute by construction.

    It thus remains to show that adding a unimodular lattice of signature 0 (mod 8) to $\Lambda'$ yields isomorphic tuples $(G,\hat q,f)$ via the construction in Lemma~\ref{lem:PSM-ASCS-ess-surj}.

    Consider then the lattice $\Lambda = \Lambda'\oplus U$ and the corresponding discriminant group $G_{\Lambda_0}$. A $W_2$ class on $\Lambda$ splits as $W_2 = W_2' + W_2^U$. Note that due to the properties of the Gauß sum on discriminant groups and unimodularity of $U$ we know that $(W_2^U,W_2^U)\equiv \sigma^U \equiv 0$ modulo 8, see~\cite[Prop.\,5.42]{HS05}. In particular we know that $\frac{W_2^U}{2}$ is even. The kernel $\Lambda_0$ splits into $\Lambda_0 = \Lambda'_0\oplus U_0 + (\Lambda'\setminus\Lambda'_0\oplus U\setminus U_0)$. Choosing generators of $\Lambda'/\Lambda_0'$ and $U/U_0$ as $f'$ and $f_U$ we may rewrite this as $\Lambda_0 = \Lambda_0' + U_0 + \langle f \rangle$ with $f := f' + f_U$, where we use $\langle\cdot\rangle$ to mean the $\Zb$-span in the ambient $\Qb^N$. Next consider the dual $\Lambda_0^*$. As above it is generated as ${\Lambda}^*+\langle \frac{W_2}{2}\rangle = {(\Lambda')}^* + U + \langle\frac{W_2}{2}\rangle$.

    We now calculate
    \begin{align*}
        \frac{{(\Lambda_0)}^*}{\Lambda_0} &= \frac{{(\Lambda')}^* + U +\langle\frac{W_2}{2}\rangle}{\Lambda'_0 + U_0+\langle f\rangle}\\
        &\overset{(1)}{\cong} \frac{{(\Lambda')}^* + \langle f_U\rangle + \langle\frac{W_2'}{2}\rangle}{\Lambda_0'+\langle f\rangle}\\
        &\overset{(2)}{\cong} \frac{{(\Lambda')}^*+\langle\frac{W_2'}{2}\rangle}{\Lambda'_0} = \frac{{(\Lambda_0')}^*}{\Lambda_0'}.
    \end{align*}
    Here the isomorphisms $(1)$ stems from dividing out the common subgroup $U_0$ using $U=U_0+\langle f_U\rangle$, and the isomorphism $(2)$ comes from dividing out $\langle f\rangle$, using $(\Lambda')^*+\langle f_U\rangle = (\Lambda')^*+\langle f\rangle$ since $f=f'+f_U$ and $f'\in(\Lambda')^*$. Note that along this equivalence the induced quadratic form is preserved as we only divide by even elements.

    The left and right diagonal maps therefore lift to the left diagram.
\end{proof}
\begin{lemma}
    We have $\mathrm{ker}(F:\mathrm{Aut}(G)\to\mathrm{Aut}(\mathscr{D}))\simeq \Zb_2$.
\end{lemma}
\begin{proof}
    First we show that for any $\phi\in F^{-1}(\id_\mathscr{D})$ we have $\phi\big|_{G_0}=\id_{G_0}$. By definition $\phi$ must fix $f$ and by being a lift of the identity on $\mathscr{D}$ it must also fix the class $[g]\in\mathscr{D}$ in of any $g\in G_0$. Hence there is a map $\mu:G_0\to\{0,1\}$ such that $\phi(g) = g + \mu(g)f$. But since we must have $\hat q = \hat q\circ\phi$ we can compute, for all $g \in G_0$,
    $$
        \hat q(g) = \hat q(\phi(g)) = \hat q(g + \mu(g)f)
        = \hat b(g,\mu(g)f) + \hat q(g) + \hat q(\mu(g)f) \ .
    $$
    Using $\hat b(g,\mu(g)f)=0$ and $\hat q(\mu(g)f) = \frac12 \mu(g)$, we conclude $\mu(g)=0$ for all $g$.

    For an element $g\in G\setminus G_0$ consider the difference $\phi(g)-g\in G_0$. For any $g_0\in G_0$ we have
    \[
        \hat b(g_0,\phi(g)-g) = \hat b(g_0,\phi(g)) - \hat b(g_0,g) = \hat b(\phi(g_0),\phi(g))-\hat b(g_0,g) = 0\,,
    \]
    where we use that $\phi$ fixes $g_0$ and preserves the bilinear form. It follows that $\phi(g)-g$ is transparent with respect to $\hat b$. From the theory of spin modular categories we know that this means that $\phi(g)-g$ is either $0$ or $f$, see e.g.~\cite[Lem.\,3.5]{Bla03}.

    Once again we find that $\phi(g) = g + \nu(g)f$ for some $\nu:G\setminus G_0\to\{0,1\}$. But since for any $g,g'\in G\setminus G_0$ we have $g-g'=\phi(g)-\phi(g')\in G_0$ we find $\nu(g)=\nu(g')$. This leaves at most two possibilities for $\phi$, and these are indeed both automorphism of $(G,\hat q,f)$. Explicitly, $F^{-1}(\id_\mathscr{D})=\{\id_G,\varphi\}$ where $\varphi$ is given by $\varphi(x)=x+2b(x,f)f$.
\end{proof}

This completes the proof of Proposition~\ref{prop:spin-abelian-Chern-Simons}.

\medskip

Having established that the underlying data for both the gauging construction and the classification in~\cite{BM05} agree we will now compare the theories themselves. For the dimension of the state spaces in loc.\ cit.\ (Sec. 5) a basis of states indexed by $g$ copies of the discriminant group $\mathscr{D}$ on a genus $g$ surface is given. The ungauged state spaces in our theory have dimension ${|G|}^g$ and as we pointed out in Example~\ref{ex:pointed_spin_state_space}\,(1) they split evenly along the $2^{2g}$ spin structures. Thus we also arrive at a dimension of $\frac{{|G|}^g}{4^g}={|\mathscr{D}|}^g$.

We may further compare the representations of the spin mapping class group induced by the two theories. For simplicity we will only look at the unpunctured torus here. We compute the mapping class group action on the gauged spin theory by applying Theorem~\ref{thm:spin_refinement}. In the oriented case the (projective) action is simply given by the $S$- and $T$-matrix. It is constructed by evaluating the bordism $M=T\times[0,1]\cup_\varphi T\times\{1\}$, where $\varphi$ is a representative of the corresponding mapping class group element. The spin version works in the same way, but $\varphi$ will in general change the spin structure on $\Sigma$. The induced representation of the mapping class group is again constructed by evaluation of the glued (spin) bordism $M$ under the TFT. It is here that Theorem~\ref{thm:spin_refinement} tells us that summing over the spin structures on $M$ will yield the oriented $S$- and $T$-matrices. In particular we have that the state spaces embed as
\[
    \Zspin(T^2,\sigma)\subset \Zor(T^2;A)=\Zor(T^2)\,,
\]
where the equality follows from the fact that the $f$-punctured torus has a zero-dimensional state space. The restriction of the bordisms $M_{T,S}$ implementing the $T$- and $S$-transformations on the torus split as
\[
    \Zor(M_{T,S})\big|_{\Zspin(T^2,\sigma)} = \frac{1}{2} \big( \Zspin(M_{T,S},s_1) + \Zspin(M_{T,S},s_2)\big)\,,
\]
where $s_1$ and $s_2$ are the two spin structures that can be put on $M_{T,S}$ for a given fixed spin structure on the boundary. A priori these may differ, but the difference only shows in the part of the spin state space arising from the $f$-puncture, which is trivial in the case here. We thus are simply left with
\[
    \Zor(M_{T,S})\big|_{\Zspin(T^2,\sigma)} = \Zspin(M_{T,S},s_1)\,.
\]
Extracting the mapping class group action is now a simple matter of linear algebra. There are four spin structures on the torus which may be given in terms of the difference to the standard spin structure on the handle-body in terms of the symplectic basis as a tuple in $(\alpha_i,\beta_i)$, see e.g.~\cite{Bel98} for a closer description. Recalling the projectors from Figure~\ref{fig:state_space_projector}, it is easy to see that a basis of the state spaces with $\alpha_i=0$ is given by the span of $\mathrm{coev}_g + \mathrm{coev}_{f \otimes g} \in \Cc(1,g\otimes g^*\oplus (f\otimes g)\otimes{(f\otimes g)}^*)\simeq k^2$, and for $\alpha_i=1$ is given by the span of $\mathrm{coev}_g - \mathrm{coev}_{f \otimes g}$, where $g$ runs over a set of representatives of $G_0/\langle f\rangle$ if $\beta_i=0$ and over $G_1/\langle f\rangle$ if $\beta_i=1$, where $G_1:=G\setminus G_0$. Denote these basis vectors $e_g^\pm$.

It remains to write the $T$- and $S$-matrices in the basis $e_g^\pm$. The $T$-matrix is diagonal in the oriented theory with the standard basis for the state space. It is easy to determine that in the basis $e_g^\pm$ the $T$-matrix acts as
\begin{align*}
    e_g^+ &\mapsto e^{-2\pi i\hat q(g)} e_g^-\>, & e_g^- &\mapsto e^{-2\pi i\hat q(g)} e_g^+\>, & &(g\in G_0)\\
    e_g^+ &\mapsto e^{-2\pi i\hat q(g)} e_g^+\>, & e_g^- &\mapsto e^{-2\pi i\hat q(g)} e_g^-\>. & &(g\in G_1)
\end{align*}
For the $S$-matrix one finds
\begin{align*}
    e_g^+ &\mapsto \frac{2}{|G|^\frac{1}{2}}\sum_{[g']\in G_0/\langle f\rangle} e^{2\pi i\hat b(g,g')} e_{g'}^+\>, & e_g^- &\mapsto \frac{2}{|G|^\frac{1}{2}}\sum_{[g']\in G_1/\langle f\rangle} e^{2\pi i\hat b(g,g')} e_{g'}^+\>, & &(g\in G_0)\\
    e_g^+ &\mapsto \frac{2}{|G|^\frac{1}{2}}\sum_{[g']\in G_0/\langle f\rangle} e^{2\pi i\hat b(g,g')} e_{g'}^-\>, & e_g^- &\mapsto \frac{2}{|G|^\frac{1}{2}}\sum_{[g']\in G_1/\langle f\rangle} e^{2\pi i\hat b(g,g')} e_{g'}^-\>. & &(g\in G_1)
\end{align*}

To compare this to the results of~\cite{BM05} we need to introduce some of their notation. There, a basis of the state spaces is given by the physical wave functions
\[
    \Big(\Psi_{\gamma,\frac{\beta}{2},\frac{\alpha}{2}}^\text{phys}\Big)_{\gamma\in \mathscr{D}}\>,
\]
where $\alpha$ and $\beta$ encode the spin structure on the torus. (Note that their convention for the order of $\alpha$ and $\beta$ is reversed from ours.) Further they denote an operator $U$ going from a torus with spin structure $\beta_1,\alpha_1$ to one with spin structure $\beta_2,\alpha_2$ by $U\text{\tiny$\setlength{\arraycolsep}{2pt} \begin{bmatrix}\beta_2 & \alpha_2\\ \beta_1 & \alpha_1\end{bmatrix}$}$. Then the matrices of~\cite[Sec.\,5.6]{BM05} can be written as\footnote{Note also that these are not precisely the matrices as given in~\cite[Sec.\,5.6.1]{BM05}. The physical wave functions $\Psi^\mathrm{phys}$ actually depend on a choice of representative of $[\frac{\alpha}{2}]\in\frac{1}{2}\Zb/\Zb$ (and the same for $\beta$). Changing the choice will incur an additional phase, and is more convenient for as to make the choices as given above. In practice only $T\text{\tiny$ \setlength{\arraycolsep}{2pt} \begin{bmatrix} 0 & 1\\ 0 & 0\end{bmatrix}$}$ is different.}:
\begin{align*}
    T\text{\tiny$\setlength{\arraycolsep}{2pt} \begin{bmatrix} 0 & 1\\ 0 & 0\end{bmatrix}$}_\gamma^{\gamma'} &= \delta_\gamma^{\gamma'} e^{2\pi i\sigma/24 - 2\pi i [q(\gamma) - q(0)]},\\
    T\text{\tiny$\setlength{\arraycolsep}{2pt} \begin{bmatrix} 0 & 0\\ 0 & 1\end{bmatrix}$}_\gamma^{\gamma'} &= \delta_\gamma^{\gamma'} e^{2\pi i\sigma/24 - 2\pi i [q(-\gamma) - q(0)]},\\
    T\text{\tiny$\setlength{\arraycolsep}{2pt} \begin{bmatrix} 1 & 0\\ 1 & 0\end{bmatrix}$}_\gamma^{\gamma'} = T\text{\tiny$\setlength{\arraycolsep}{2pt} \begin{bmatrix} 1 & 1\\ 1 & 1\end{bmatrix}$}_\gamma^{\gamma'} &= \delta_\gamma^{\gamma'} e^{2\pi i\sigma/24 - 2\pi i q(-\gamma)}
\end{align*}
and
\begin{align*}
    S\text{\tiny$\setlength{\arraycolsep}{2pt} \begin{bmatrix} 0 & 0\\ 0 & 0\end{bmatrix}$}_\gamma^{\gamma'} = S\text{\tiny$\setlength{\arraycolsep}{2pt} \begin{bmatrix} 0 & 1\\ 1 & 0\end{bmatrix}$}_\gamma^{\gamma'} &= |\mathscr{D}|^{-\frac{1}{2}}e^{2\pi i b(\gamma,\gamma')},\\
    S\text{\tiny$\setlength{\arraycolsep}{2pt} \begin{bmatrix} 1 & 0\\ 0 & 1\end{bmatrix}$}_\gamma^{\gamma'} &= |\mathscr{D}|^{-\frac{1}{2}}e^{2\pi i b(\gamma,\gamma'+W_2)},\\
    S\text{\tiny$\setlength{\arraycolsep}{2pt} \begin{bmatrix} 1 & 1\\ 1 & 1\end{bmatrix}$}_\gamma^{\gamma'} &= |\mathscr{D}|^{-\frac{1}{2}}e^{2\pi i b(\gamma,\gamma'+W_2) + 4\pi i q(0)}.
\end{align*}
The factors of $e^{2\pi i \sigma/24}$ in the $T$-matrix come from deprojectifying the representation of the mapping class group which we will ignore those.

Then identifying $2a=W_2$ and sending
\begin{align*}
    e^+_g &\mapsto \Psi_{[g],0,0}^\text{phys}\>, & e^-_g &\mapsto e^{-2\pi i \hat b(a,g)}\Psi_{[g],0,\frac{1}{2}}^\text{phys}\>, & (g\in G_0)\\
    e^+_g &\mapsto \Psi_{[g-a],\frac{1}{2},0}^\text{phys}\>, & e^-_g &\mapsto e^{-2\pi i \hat b(a,g)}\Psi_{[g-a],\frac{1}{2},\frac{1}{2}}^\text{phys}\>, & (g\in G_1)
\end{align*}
yields an isomorphism of the mapping class group representations.

\newpage

\printbibliography

\end{document}